\documentclass[11pt,a4paper,leqno,fullpage]{article}
\usepackage[usenames,dvipsnames]{color}
\usepackage{amsmath, amssymb, amscd, amsthm, amsfonts, amsbsy, appendix}
\usepackage{bbm, stmaryrd, mathrsfs, bm, linearb, upgreek, skak, textcomp, mathdots}
\usepackage{mathtools, titlesec, extarrows, tocloft, url, dsfont, verbatim, latexsym}
\usepackage[intoc]{nomencl}
\usepackage[nottoc]{tocbibind}
\usepackage[mathscr]{euscript}
\usepackage[all, cmtip]{xy} 
\usepackage{graphicx, pgf}
\usepackage[T1]{fontenc}

\usepackage{hyperref}

\usepackage{hyperref}

\usepackage{a4}

\newenvironment{myeq}[1][]
{\stepcounter{theorem}\begin{equation}\tag{\thetheorem}{#1}}
{\end{equation}}

\newenvironment{mysubsection}[2][]
{\begin{subsec}\begin{upshape}\begin{bfseries}{#2.}
\end{bfseries}{#1}}
{\end{upshape}\end{subsec}}

\newcommand{\uM}{\underline{M}}

\newcommand{\uN}{\underline{N}}
\newcommand{\uQ}{\underline{\mathbb{Q}}} 
\newcommand{\uQZ}{\underline{\mathbb{Q}/\Z}}

\newcommand{\Z}{\mathbb Z}
\newcommand{\JJ}{\mathcal J}
\newcommand{\RR}{\mathcal R}
\newcommand{\bZ}{\langle \Z \rangle}
\newcommand{\bbZ}{\langle \langle \Z \rangle \rangle}

\newcommand{\bZp}{\langle \Z/p \rangle}
\newcommand{\bZpi}{\langle \Z/p_i \rangle}

\newcommand{\bZpk}{\langle \Z/p_k \rangle}
\newcommand{\bZq}{\langle \Z/q \rangle}

\newcommand{\BB}{\mathcal B}
\newcommand{\MM}{\mathcal{M}}
\newcommand{\CC}{\mathcal{C}}

\newcommand{\KK}{\mathcal{K}}

\newcommand{\UU}{\mathcal{U}}
\newcommand{\II}{\mathcal{I}}
\newcommand{\fI}{\mathscr{I}}
\newcommand{\fS}{\mathscr{S}}
\newcommand{\TT}{\mathcal{T}}

\newcommand{\uZ}{\underline{\mathbb{Z}}}

\newcommand{\R}{\mathbb R}

\newcommand{\C}{\mathbb C}
\newcommand{\Q}{\mathbb Q}

\newcommand{\res}{\mathit{res}}

\newcommand{\Hom}{\mathit{Hom}}
\newcommand{\Ext}{\mathit{Ext}}

\newcommand{\coker}{\mbox{coker}}
\newcommand{\Rep}{\mathit{Rep}}
\newcommand{\uH}{\underline{H}}
\newcommand{\uA}{\underline{A}}
\newcommand{\ua}{\underline{a}}

\newcommand{\uk}{\underline{k}}
\newcommand{\upi}{\underline{\pi}}

\newcommand{\tH}{\tilde{H}}
\newcommand{\bs}{\bigstar}
\newcommand{\tHbs}{\tilde{H}^\bigstar_G}
\newcommand{\tHal}{\tilde{H}^\alpha_G}

\newcommand{\tHp}{\tilde{H}^+_G}
\newcommand{\tHpl}{\tilde{H}^{>0}_G}
\newcommand{\tHm}{\tilde{H}^-_G}
\newcommand{\tHz}{\tilde{H}^{(0)}_G}
\newcommand{\tHr}{\tilde{H}^{\Rep}_G}
\newcommand{\tHrz}{\tilde{H}^{\Rep_0}_G}
\newcommand{\uHbs}{\underline{H}^\bigstar_G}
\newcommand{\uHal}{\underline{H}^\alpha_G}
\newcommand{\uHald}{\underline{H}^{\alpha^{C_d}}_G}
\newcommand{\uHalxi}{\underline{H}^{\alpha-\xi}_G}

\theoremstyle{plain}
\newtheorem{theorem}{Theorem}[section]
\newtheorem{thm}[theorem]{Theorem}
\newtheorem{lemma}[theorem]{Lemma}
\newtheorem{cor}[theorem]{Corollary}
\newtheorem{prop}[theorem]{Proposition}

\newtheorem{subsec}[theorem]{}

\theoremstyle{definition}
\newtheorem{notation}[theorem]{Notation}
\newtheorem{defn}[theorem]{Definition}
\newtheorem{rmk}[theorem]{Remark}
\newtheorem{remark}[theorem]{Remark}

\newtheorem{exam}[theorem]{Example}

\newtheorem*{theorem*}{Theorem}

\makeatletter
\newcommand{\subjclass}[4]{%
  \let\@oldtitle\@title%
  \gdef\@title{\@oldtitle\footnotetext{2010 \emph{Mathematics subject classification.} #1, #2 (primary); #3, #4 (secondary).}}%
}
\newcommand{\keywords}[4]{%
  \let\@@oldtitle\@title%
  \gdef\@title{\@@oldtitle\footnotetext{\emph{Key words and phrases.} #1, #2, #3, #4.}}%
}

\makeatother

\title{ Equivariant cohomology for cyclic groups of square-free order}
\author{Samik Basu, Surojit Ghosh}
\date{ }

{

\begin{document}
 {55P91}{57S17}{55Q91}
 {Mackey functor}{freeness theorem}{Grassmann manifolds} 

  \maketitle
 \begin{abstract}

The main objective of this paper is to compute $RO(G)$-graded cohomology of $G$-orbits for the group $G=C_n$, where $n$ is a product of distinct primes. We compute these groups for the constant Mackey functor $\uZ$ and for the Burnside ring Mackey functor $\uA$. Among other things, we show that the groups $\uHal(S^0)$ are mostly determined by the fixed point dimensions of the virtual representations $\alpha$, except in the case of $\uA$ coefficients when the fixed point dimensions of $\alpha$ have many zeros. In the case of $\uZ$ coefficients, the ring structure on the cohomology is also described. The calculations are then used to prove freeness results for certain $G$-complexes. 
\end{abstract}

\section{Introduction}

The celebrated  solution to the Kervaire invariant one problem \cite{HHR16} demonstrated how techniques in equivariant stable  homotopy theory may be used in the computation of homotopy groups. One of the steps in the computation involves calculating Bredon cohomology groups for the group $C_8$. This has given rise to a renewed interest in equivariant cohomology computations after these were initially defined \cite{Bre67} some years back. 

Bredon cohomology was defined as a natural generalization of ordinary cohomology in the context of obstruction theory. This is graded over integers and was defined for {\it ``coefficient systems''} \cite[Section I.4]{May96}, and they naturally carried suspension isomorphisms $\tH^k_G(X;\uM) \cong \tH^{k+n}_G(\Sigma^n X; \uM)$ for based $G$-spaces $X$. However, the integer graded cohomology theory proved to be inadequate for Poincar\'{e} duality. 

Let $G$ be a finite group and $V$ an orthogonal $G$-representation. One defines $S^V$ to be the one-point compactification of $V$, and uses the orthogonal structure to define $G$-spaces 
$$D(V) :=\{ v \in V : ||v|| \leq 1\}$$ 
and  
$$S(V) :=\{ v \in V : ||v|| =1\}.$$
The $G$-sphere $S^V$, which is $G$-homeomorphic to $D(V)/S(V)$, has a $G$-fixed basepoint (the $\infty$ point). The equivariant stable homotopy category \cite{LMS86} is constructed so that the smash product with $S^V$ (denoted by $\Sigma^V$) is an isomorphism. The category then carries certain strong change of groups isomorphisms such as the Wirthm\"{u}ller isomorphism and the Adams' isomorphism (see \cite{Ada84}, and \cite[Chapter II]{LMS86}). One of the models representing this homotopy category is equivariant orthogonal spectra \cite{MM02} (see also \cite{Sch18} for a slightly different but equivalent category). We write $\{-,-\}^G$ to denote the maps in the equivariant stable homotopy category. 

Let $E$ be a $G$-spectrum.\footnote{From this point onwards, this means an orthogonal $G$-spectrum.} Then, the cohomology theory represented by $E$ is graded over $RO(G)$, the representation ring of $G$ over $\R$. For $\alpha \in RO(G)$ written as $V - W$, we write $S^\alpha = \Sigma^{-W} S^V$, and define $E^\alpha(X) $ as $\{X, S^\alpha \wedge E\}^G$. The analogue of {\it `` ordinary cohomology''} theories in this context are the cohomology theories represented by spectra whose (integer-graded) homotopy groups are concentrated in degree $0$. These carry more structure than {\it ``coefficient systems''} and are called Mackey functors \cite{Dre72}. In fact, for any $G$-Mackey functor $\uM$, there is an equivariant Eilenberg-MacLane spectrum $H\uM$ \cite{LMM81}, and this gives a cohomology theory $X \mapsto \tH^\alpha_G(X;\uM)$ on based $G$-spaces $X$. One may further extend this construction so  that there is a Mackey functor $\uHal(X;\uM)$. The groups in grading $n-V$ compute the integer-graded Bredon cohomology of $S^V$, and those in grading $V-n$ compute the integer-graded Bredon homology of $S^V$.  

The category of $G$-Mackey functors has a symmetric monoidal structure $\Box$, with the unit being the Burnside ring Mackey functor $\uA$. Another important Mackey functor from the point of view of computations is $\uZ$ which is also a monoid under $\Box$. In these cases, the $RO(G)$-graded cohomology is a graded commutative ring. Remarkably, the calculation of $\uHal(S^0; \uA)$ and $\uHal(S^0;\uZ)$ have proved to be difficult in general.  For the group $C_p$, the additive structure was described by Stong and the ring structure by Lewis \cite{Lew88}. For the group $C_p$ and the coefficients $\underline{\Z/p}$, the ring structure was described by Stong  (see Appendix of \cite{Car00}). For the group $C_{pq}$, the additive structure for $\uA$ and $\uZ$ coefficients is described in \cite{BG19}. There have been significant partial computations for $\uZ$ coefficients for the group $C_{p^n}$, and in the case of other groups. In this paper, we make calculations for $G=C_n$, where $n$ is odd and square free. Most of our results also work if $n$ is even and the sign representation occurs an even number of times. In the following, we describe our results in more detail. 

\begin{mysubsection}{Computations for $\uH^\alpha_{C_n}(S^0;\uZ)$}
The complex irreducible representations of $C_n$ are given by homomorphisms from $C_n$ to $S^1$, and they are of the form $\xi^r$ ($1\leq r \leq n$) where $\xi$ is the multiplication by the $n^{th}$-roots of $1$. The ring $RO(C_n)$ has a basis given by the realizations of $\xi^r$, for which we retain the same notation. In the computation of $\uH^\alpha_{C_n}(S^0;\uZ)$, we describe the Mackey functor for every $\alpha$ that is a linear combination of $\xi^r$ for $1\leq r \leq n$. One has the easily proved equivalence 
$$H\uZ \wedge S^{\xi^r} \simeq H\uZ \wedge S^{\xi^{rs}}$$ 
if $s$ is relatively prime to $n$. This equivalence implies that it suffices to compute only for $\alpha$ which is a linear combination of $\xi^d$ for $d | n$. 

The next observation is that a linear combination of $\xi^d$ for $d | n$ is determined by the collection of fixed point dimensions $(|\alpha^{C_d}| ~|~ d | n)$. It turns out that the formula for the $RO(C_n)$-graded cohomology has a convenient expression when described according to the fixed point dimensions. We compute \eqref{Zval}
$$\tH^\alpha_{C_n}(S^0; \uZ) = \begin{cases} \Z & \text{if} \; |\alpha| =0 \\ 
                                                 \Z/m(\alpha) & \text{if} \; |\alpha| >0 \mbox{ even} \\ 
                                                \Z/m(\alpha) & \text{if} \; |\alpha| <0 \mbox{ odd} \\ 
                                                            0 & \text{otherwise,} \end{cases} $$
where $m(\alpha)$ is an integer depending on the fixed point dimensions of $\alpha$ (see Definition \ref{malph}). This describes the cohomology groups with $\uZ$-coefficients. The Mackey functor values are described as  (Theorems \ref{zero},\ref{non-zero}, and \ref{zercoh})
$$\uH^\alpha_{C_n}(S^0; \uZ) = \begin{cases} \uZ^{\JJ(\alpha)} & \text{if} \; |\alpha| =0 \\ 
                                                 \bigoplus_{|\alpha^{C_{p_i}}|\leq 0}\KK_i\bZpi. & \text{if} \; |\alpha| >0 \mbox{ even} \\ 
                                                \bigoplus_{|\alpha^{C_{p_i}}|>1} \KK_i\bZpi & \text{if} \; |\alpha| <0 \mbox{ odd} \\ 
                                                            0 & \text{otherwise.} \end{cases} $$

For the ring structure, we note that the even degree classes are in non-negative total degree. They form a subring which we denote by $\tH^+_{C_n}(S^0;\uZ)$. The odd degree classes are in negative total degree and are denoted by $\tH^-_{C_n}(S^0;\uZ)$. This forms a square $0$ module over $\tH^+_{C_n}(S^0;\uZ)$. We describe the ring structure by writing $\tH^+_{C_n}(S^0;\uZ)$ via generators and relations, and then describing $\tH^-_{C_n}(S^0;\uZ)$ as a module over it. 

One defines classes $a_V \in \tH^V_{C_n}(S^0;\uZ)$ and $u_V \in \tH^{V - \dim(V)}_{C_n}(S^0;\uZ)$ \cite{HHR16} for $C_n$-representations $V$, which satisfies 
$$a_Va_W=a_{V\oplus W},~ u_Vu_W=u_{V\oplus W}.$$
One also has the relations 
$$\frac{n}{d} a_{\xi^d}=0, ~ \frac{d}{\gcd(d,s)}a_{\xi^s}u_{\xi^d}= \frac{s}{\gcd(d,s)}u_{\xi^s}a_{\xi^d}.$$
In terms of these generators, we prove 
\begin{thm}
The ring $\tH^+_{C_n}(S^0;\uZ)$ is a subring of $\Z[u_{\xi^d}^{\pm}\mid d| n]\otimes \Z[a_\xi]/(na_\xi)$. 
\end{thm}
In fact, we also clearly determine the image of the degree-wise inclusions in the theorem above (Theorems \ref{ringz} and \ref{ringp}). For $\alpha$ such that $|\alpha|=0$, both 
$\tH^\alpha_{C_n}(S^0; \uZ)$ and $\Z[u_{\xi^d}^{\pm}\mid d| n]\otimes \Z[a_\xi]/(na_\xi)$ at degree $\alpha$ is isomorphic to $\Z$, and the inclusion is multiplication by $\frac{n}{m(\alpha)}$. In the case $|\alpha|>0$, the left hand side is $\Z/m(\alpha)$, the right hand side is $\Z/n$, and the inclusion is also multiplication by $\frac{n}{m(\alpha)}$. 

We denote by $\tH^{>0}_{C_n}(S^0;\uZ)$ the part of $\tH^\bs_{C_n}(S^0;\uZ)$ in positive total grading. This is $\tH^+_{C_n}(S^0;\uZ)$-module, and we prove (Theorem \ref{ringodd})
\begin{theorem} 
As a $\tH^+_{C_n}(S^0;\uZ)$-module, the negative part of $\tH^\bs_{C_n}(S^0;\uZ)$ (denoted $\tH^-_{C_n}(S^0;\uZ)$) is isomorphic to $\Sigma^{3-\xi}\Hom(\tH^{>0}_{C_n}(S^0;\uZ),\Q/\Z)$.
\end{theorem}
One is usually also interested in understanding the part of $\tH^\bs_{C_n}(S^0;\uZ)$ in gradings $V-n$ for $V$ a $G$-representation and $n \in \Z$. We denote this as $\tH^{Rep}_{C_n}(S^0;\uZ)$, and prove (Theorem \ref{ringrep}) 
\begin{thm}
The ring $\tHr(S^0;\uZ)$ is generated by the classes $u_{\xi^d}$, $a_{\xi^d}$ for divisors $d$ of $n$ such that $d\neq n$, subject to the relations 
$$ \frac{n}{d}a_{\xi^d}=0, ~~ \frac{d}{\gcd(d,s)}a_{\xi^s}u_{\xi^d}= \frac{s}{\gcd(d,s)}u_{\xi^s}a_{\xi^d}.$$
\end{thm} 

\end{mysubsection}

\begin{mysubsection}{Computations for $\uH^\alpha_{C_n}(S^0;\uA)$} 
We compute the additive structure of the Mackey functor valued cohomology groups $\uH^\alpha_{C_n}(S^0;\uA)$ in most cases. For the Eilenberg-MacLane spectrum $H\uA$, we no longer have the equivalence of $H\uA \wedge S^{\xi^r}$ and $H\uA \wedge S^{\xi^{rs}}$, so now we would have to compute $\uH^\alpha_{C_n}(S^0;\uA)$ for the full $RO(C_n)$-grading and not just the linear combination of $\xi^d$ for divisors $d$ of $n$. However, we prove that the groups $\tH^\alpha_{C_n}(S^0;\uA)$, up to isomorphism, depend only on the fixed point dimensions of $\alpha$ (Corollary \ref{mzcalc}). 
\begin{thm}
Let $n=p_1\cdots p_k$. Then, 
$$\tH^{\alpha}_{C_n}(S^0;\uA) \cong \Z^{\# \{d|n ~\mid ~|\alpha^{C_d}|=0 \}} \oplus \bigoplus_{i=1}^k (\Z/p_i)^{\# \{d|\frac{n}{p_i} ~\mid ~|\alpha^{C_d}|>0, ~|\alpha^{C_{dp_i}}| <0 \}}.$$
\end{thm}
The Mackey functor valued cohomology $\uH^\alpha_{C_n}(S^0;\uA)$ does in fact, depend on the actual representation. This has been observed in \cite{Lew88} for the group $C_p$, and in \cite{BG19} for the group $C_{pq}$. We say that $\alpha$ is {\it non-zero} if all the fixed point dimensions are non-zero, and {\it mostly non-zero} if $|\alpha^{C_d}|=0$ implies that $|\alpha^{C_{dp}}| \neq 0$ for divisors $p$ of $n$ such that $p \nmid d$. We prove that the Mackey functors $\uH^\alpha_{C_n}(S^0;\uA)$ depend only on the fixed point dimensions of $\alpha$, if $\alpha$ is mostly non-zero. Examples from \cite{BG19} demonstrate that if $\alpha$ is not mostly non-zero, the Mackey functor does not depend solely on the fixed point dimensions. This is proved by showing that $\uH^\alpha_{C_n}(S^0;\uA)$ decomposes into sum of Mackey functors related to cohomology with constant coefficients. In the non-zero case, we prove (Theorem \ref{coh_ar})
\begin{thm}
For $\alpha$ non-zero, there is an isomorphism 
$$\uH^\alpha_{C_n}(S^0;\uA) \cong \bigoplus_{d\mid n} \bZ_d\boxtimes \uH^{\alpha^{C_d}}_{C_{\frac{n}{d}}}(S^0;\uZ).$$ 
\end{thm} 
In the mostly non-zero case, we prove (Theorem \ref{indzero}) 
\begin{thm}
Let $n=p_1\cdots p_k$. For $\alpha$ mostly non-zero, there is an isomorphism 
$$\uH^\alpha_{C_n}(S^0;\uA) \cong \bigoplus_{d\mid n} \bZ_d\boxtimes \uH^{\alpha^{C_d}}_{C_{\frac{n}{d}}}(S^0;\uZ)[\zeta_\alpha(\{i~|~p_i|d\})].$$ 
\end{thm}
It follows that we have an  {\it ``independence theorem''}
\begin{thm}
1) For all $\alpha \in RO(C_n)$, the groups $\tH^\alpha_{C_n}(S^0;\uA)$ depend only on the fixed point dimensions of $\alpha$. \\
2) For $\alpha$ mostly non-zero, the Mackey functors $\uH^\alpha_{C_n}(S^0;\uA)$ depend, up to isomorphism, only on the fixed point dimensions of $\alpha$. 
\end{thm}
\end{mysubsection}

\begin{mysubsection}{Freeness theorems}
The cohomology $\uHbs(X_+;\uA)$ of a $G$-space $X$ is a module over $\uHbs(S^0;\uA)$. As an application of the calculation of $\uHbs(S^0;\uA)$, we provide conditions under which the cohomology is a free module. In the non-equivariant case, one knows that if $X$ has even dimensional cells, the cohomology is torsion-free. In the equivariant case, one considers cell complexes with cells formed out of disks in orthogonal representations. For the group $C_p$, it was shown \cite{Lew88} that the cohomology is free if all the cells are even dimensional, and the cells are attached under some additional fixed point dimension restrictions. The proof involved showing that the attaching maps induce the zero map on  cohomology. The additional restrictions were removed in \cite{Fer00,FL04} where the attaching maps were not zero, yet the resulting space had free cohomology (with a basis unrelated to the cell structure). 

The freeness result of \cite{Lew88} was generalized in \cite{BG19} for the group $G=C_{pq}$. We further generalize this for the group $C_n$ (Theorem \ref{free}). 
\begin{thm}
Suppose that $X$ is a $C_n$-cell complex with only even cells satisfying the condition : if $G\times_H D(V)$ is attached after $G\times_K D(W)$, then $W \ll V$. Then, $\uH^\bs_{C_n}(X_+;\uA)$ is a free $\uH^\bs_{C_n}(S^0;\uA)$-module with one generator for each cell of $X$. 
\end{thm}
 From \cite[Remark 2.2]{Lew92} and \cite{FL04}, we know that generalizations of the freeness theorem do not exist for the  groups $C_{p^2}$ and $C_p\times C_p$, and hence, any groups containing these. Thus, this completely answers the freeness question among Abelian groups. As an application, we verify that the hypothesis is satisfied by projective spaces and the Grassmanians over certain $C_n$-representations (Theorems \ref{cohcp} and \ref{cohgr}). 
\end{mysubsection}

\begin{mysubsection}{Organization of the paper}
We start by introducing certain preliminaries on equivariant Eilenberg-MacLane spectra and $RO(G)$-graded cohomology in Section \ref{EqEM}. In section \ref{GMack}, we introduce the Mackey functors which turn up in the calculation, and study the set of maps between them in certain cases. Section \ref{cohrep} proves some general results about torsion in $\uHal(S^0)$ of order prime to $|G|$, and the cohomology of spheres in the $1$-dimensional complex $G$-representations. In Section \ref{cohZ}, we describe the computation for the constant Mackey functor $\uZ$. First, the additive structure is computed, and then the ring structure is described. In Section \ref{cohA}, the additive structure of the cohomology with $\uA$ coefficients is computed. Firstly, it is computed in the non-zero fixed point case, and then in the case where some fixed points are zero. Finally, the group values are described in the case when many fixed points are zero-dimensional. Section \ref{freenessthm} proves freeness results for certain $G$-complexes with applications to Grassmannians and projective spaces. 
\end{mysubsection}

\begin{notation}
Throughout the document, $G$ denotes the group $C_n$ with generator $\rho $, for $n=p_1\cdots p_k$ where the $p_i$ are distinct odd primes. The notation $\xi$ stands for the $2$-dimensional representation of $G$ in which the chosen generator of $G$ acts by the rotation of angle $\frac{2\pi}{n}$. An element $\alpha \in RO(G)$ is called {\it even} (respectively, {\it odd}) if its dimension $|\alpha|$ is even (respectively, odd).\footnote{Since $G=C_n$ for $n$ odd, $\alpha$ even (respectively, odd) implies that all the fixed point dimensions are even (respectively, odd).}  
\end{notation}

\begin{mysubsection}{Acknowledgements} 
The research of the first author was supported by the SERB MATRICS grant 2018/000845. 
\end{mysubsection}

\section{Equivariant Eilenberg-MacLane spectra}\label{EqEM}
In this section, we recall certain definitions and notations from equivariant homotopy theory and equivariant stable homotopy theory, with a view toward laying down the important ideas and constructions behind the equivariant Eilenberg-MacLane spectra for use in later sections. 

The category of based $G$-spaces and $G$-equivariant maps is denoted by $G\TT$. We also have an enriched category $\TT_G$ with based $G$-spaces and all maps with $G$ acting on the mapping spaces by conjugation. We use the notation $[ -, -]^G$ to denote the $G$-homotopy classes of maps between two $G$-spaces. The $G$-CW-complexes are defined as spaces obtained by iteratively attaching cells of the type $G/H\times D^m$ along an attaching map out of $G/H\times \partial D^m$. As an example, $S(\xi)$ is the $G$-CW-complex 
$$G/e\times D^0\cup_f G/e \times D^1$$
with $ f:G/e\times \{\pm 1\} \to G/e$ given by $(g,-1)\mapsto g$, $(g,1)\mapsto \rho g$. Analogously, $S(\xi^d)$ is the $G$-CW-complex 
$$G/C_d\times D^0\cup_f G/C_d \times D^1$$
with $ f:G/C_d\times \{\pm 1\} \to G/C_d$ given by $([g],-1)\mapsto [g]$, $([g],1)\mapsto \rho [g]$. As a map in the equivariant stable homotopy category, $f: {G/C_d}_+ \to {G/C_d}_+$ is $\simeq \rho - id$. Hence, we have a cofibre sequence 
\begin{myeq}\label{cofd} 
{G/C_d}_+ \stackrel{\rho - id}{\to} {G/C_d}_+ \to S(\xi^d)_+.
\end{myeq}

\begin{mysubsection}{Equivariant orthogonal spectra} The construction of the equivariant stable homotopy category referred to above for the cofibre sequences such as \eqref{cofd}, is the homotopy category of equivariant orthogonal spectra \cite[Chapter II]{MM02} using the positive stable model category structure. We briefly recall the definition of the objects of this category. 
\begin{defn}
1) A $G$-universe $\UU$ is a countably infinite dimensional representation of $G$ with an inner product such that $\UU$ contains countably many copies of each finite dimensional $G$-representation. \\ 
2) The category $\fI_G$ (enriched over $G\TT$) contains all finite dimensional subspaces of $\UU$ as objects, and the morphisms are given by linear isometric isomorphisms on which $G$ acts by conjugation. \\
3) An $\fI_G$-space $X$ is a $G$-functor $X: \fI_G \to \TT_G$. 
\end{defn}
An example of an $\fI_G$-space is the functor $S_G$ which sends $V$ to $S^V$. The smash product of based $G$-spaces gives an external smash product $\bar{\wedge}$ of two $\fI_G$-spaces as an $\fI_G\times \fI_G$-space. 
\begin{defn}
An orthogonal $G$-spectrum is an $\fI_G$-space $X$ together with a natural $G$-map $X\bar{\wedge} S_G \to X \circ \oplus$ such that unit and associativity diagrams of \cite[\S 1, \S 8]{MMSS01} commute.  
\end{defn}
An equivariant orthogonal spectrum $X$, thus, has based $G$-spaces $X(V)$ with $O(V)$-action together with $G$-equivariant structure maps 
$$\sigma: X(V)\wedge S^W \to X(V\oplus W)$$
that are $O(V)\times O(W)$-equivariant. One uses the notation $\fI_G\fS$ for the category of $G$-equivariant orthogonal spectra, and the notation $\{-,-\}^G$ for the homotopy classes of maps. The spectrum $S_G$ stands for the sphere spectrum which we from now on denote by $S^0$. Using smash products with $G$-spaces, we define the $G$-spectra $S^V$ for representations $V$ of $G$, and using the shift desuspension functors of \cite[Definition II.4.6]{MM02}, we define $S^\alpha$ for $\alpha \in RO(G)$. The equivariant cohomology theory represented by a $G$-spectrum $E$, therefore, carries an $RO(G)$-grading via 
$$E^\alpha(X) = \{X,\Sigma^\alpha E\}^G.$$

For a subgroup $H\leq G$, with the inclusion functor denoted by $i$, one has an induced forgetful functor $i^\ast$ on $G$-spaces whose left adjoint sends a based $H$-space $Y$ to $G_+\wedge_H Y$. For $G$-spectra, this construction is carried out in \cite[\S V.2]{MM02}, and one has (the $H$-universe being $i^*$ applied to the $G$-universe)
$$\fI_G\fS(G_+\wedge_H Y, X) \cong \fI_H\fS(Y,i^\ast X).$$
If $X$ is an $\Omega$-$G$-spectrum, then it is positive, stable fibrant. Further, if $Y$ is an $H$-CW-complex, then it is stable cofibrant, and we use an analogue of  \cite[Lemma 14.3]{MMSS01} to obtain an isomorphism on homotopy classes  
$$\{G_+\wedge_H Y, X\}^G \cong \{Y,i^\ast X\}^H.$$

We also use the construction of the fixed point  spectra \cite[\S V.3]{MM02}. In this respect, we see that for a normal subgroup $K$, write $\epsilon : G \to G/K$, and note that $\UU^K$ is a $G/K$-universe.  The $K$-fixed point spectrum construction takes an object $X \in \fI_G\fS$ to the object $X^K$ which evaluated on a subspace $W \subset \UU^K$ is $X(W)^K$. For $G/K$-spaces $Y$ one has the isomorphism \cite[Proposition V.3.10]{MM02}
$$\fI_G\fS(\epsilon^\ast Y, X) \cong \fI_{G/K}\fS(Y, X^K).$$
In addition, if we assume that $X$ is an $\Omega$-$G$-spectrum, so that it is fibrant in the model structure on $\fI_G\fS$, the equivalence passes to homotopy classes and we have 
$$\{\epsilon^\ast Y, X\}^G \cong \{Y, X^K\}^{G/K}.$$
In the above, we are again using an analogue of \cite[Lemma 14.3]{MMSS01} to compute the homotopy classes using the stable model structure where the $G$-CW-complexes are cofibrant. 
\end{mysubsection}

 \begin{mysubsection}{Mackey Functors} Equivariant Eilenberg-MacLane spectra are spectra whose integer-graded homotopy groups are non-zero only in degree $0$. Some of them may be constructed from $\Z[G]$-modules, but the more general Mackey functors also yield these spectra. 
 
Mackey functors are defined as contravariant additive functors from the Burnside category (denoted $\BB_G$) to Abelian groups (denoted $\mathcal{A} b$). If $S,T$ are finite $G$-sets we write $C(S,T)$ as the isomorphism classes of diagrams of the type
   $$\xymatrix@R-=.25cm@C-=.25cm{ &  &U \ar[dl] \ar[dr]  \\ &T  &  &S}.$$ 
Two diagrams $T \leftarrow U \to S$ and $T \leftarrow V \to S$ are isomorphic if there is a commutative diagram  
    $$\xymatrix@R-=.25cm@C-=.25cm{&  &U \ar[dl] \ar[dd]^{\cong} \ar[dr]  \\ &T  &  &S \\& &V \ar[ul] \ar[ur]}.$$
The set $C(S,T)$ is a commutative monoid under disjoint union. Using this notation, we define
\begin{defn}
The Burnside category $\BB_G$ is a category whose objects are finite $G-$sets, and the morphisms between two object $S$ and $T$ is the group completion of $C(S,T)$. The composition is defined as the pullback.  
\end{defn}
   Let  $f:S\to T$ be a $G$-map. The two morphisms $f_!$ and $f$ in $\BB_G$ as depicted in the pictures
$$\xymatrix@R-=.25cm@C-=.25cm{ &  &S \ar[dl]_{=} \ar[dr]^{f}  \\ &S  &  &T}  \; \; \text{and} \; \; \xymatrix@R-=.25cm@C-=.25cm{ &  &S \ar[dl]_{f} \ar[dr]^{=}  \\ &T  &  &S}$$
 generate all the morphisms of the Burnside category by taking disjoint union and compositions. 
\begin{defn}\cite{Dre72,GM95}
A Mackey functor $\uM $ consists of a pair $ (\uM_\ast , \uM^\ast )$ of functors from the category of finite $G$-sets to $\mathcal{A} b$ taking disjoint unions to direct sums, with $\uM_\ast$ covariant and $\uM^\ast$ contravariant. On every object $S$, $\uM^\ast(S)=\uM_\ast(S)$. For every pullback diagram  
$$\xymatrix{ P \ar[r]^\delta    \ar[d]^\gamma                             & X \ar[d]^\alpha \\ 
                        Y \ar[r]^\beta                                                     &  Z,}$$
of finite $G$-sets, the square 
$$\xymatrix{ \uM(P) \ar[r]^{\uM_\ast(\delta)}                                & \uM(X)  \\ 
                        \uM(Y) \ar[u]^{\uM^\ast(\gamma)}  \ar[r]_{\uM_\ast(\beta)}                                                     &  \uM(Z)\ar[u]_{\uM^\ast(\alpha)}}$$
commutes, that is, $\uM^\ast(\alpha) \circ \uM_\ast(\beta) = \uM_\ast(\delta) \circ \uM^\ast(\gamma).$ 
\end{defn}

Any finite $G$-set can be  written as disjoint union of orbits $G/H$, which leads to an equivalent formulation of Mackey functors is given as follows\footnote{Observe in our case $G=C_n$ is Abelian, so we have incorporated certain simplifications.}. A Mackey functor $\uM$ comprises \\ 
I) For every $H\leq G$, an Abelian group $\uM(G/H)$. \\
II) For $H\leq K$, homomorphisms $\res^K_H : \uM(G/K) \to \uM(G/H)$ (the restriction map)  and  $tr^K_H : \uM(G/H) \to \uM(G/K) $ (the transfer map). \\ 
III) For every $g\in G$,  $c_g : \uM(G/H) \to \uM(G/H)$. \\
These are required to satisfy\\
1) $tr^H_H,$ $\res^H_H$, $c_h: \uM(G/H) \to \uM(G/H)$ are identity maps  for $H \leq G$ and $h \in H$. \\ 
2) If $L\leq H \leq K$ then $tr^K_{L} = tr^K_H tr^H_L$ and $\res^K_L = \res^K_H \res^H_L.$ \\ 
3) $c_g c_h = c_{gh}$ for $g,h \in G.$\\
4) If $H\leq K$ and $g \in G$ then  $tr^K_Hc_g = c_g tr^K_H$ and $\res^K_Hc_g = c_g \res^K_H.$\\
5) $\res^K_J tr ^K_H = (\sum_{x\in [J \setminus K /H]} c_x) tr^{J}_{J\cap H}  \res^H_{J\cap H}$ 
for all subgroups $J,H \leq K.$
 
\vspace*{0.5cm}

We denote the category of $G$-Mackey functors by $\MM_G$. The objects of this category are the Mackey functors defined above, and the morphisms are Mackey functor homomorphisms (that is, groups homomorphisms which commute with all the restriction and transfer maps, and the maps $c_g$). 

The Mackey functors are contravariant functors from $\BB_G$ to Abelian groups, so the representable functors are examples of Mackey functors. The Mackey functor $\BB_G( -, G/G)$ is called the Burnside ring Mackey functor and denoted by $\uA$. For a finite $G$-set $S$, the Mackey functor $\BB_G(-,S)$ is denoted  by $\uA_S$ for any finite $G$-set $S$. The Mackey functor $\uA$ can be described in more explicit terms as $\uA(G/H)=A(H)$\footnote{$A(G)$ is the Burnside ring of $G$, which additively is the group completion of the commutative monoid of finite $G$-sets under disjoint union.}, with restriction maps given by realizing a $K$-set as an $H$-set, and transfer by the functor $K\times_H -$. If $M$ is a $G$-module, it gives a Mackey functor $\uM$ by 
$$\uM(G/H)= M^H, ~\res^K_H = \{M^K \subset M^H\},~tr^K_H(m)= \sum_{x\in K/H} xm.$$
One also notes that the category $\BB_G$ is it's own opposite. Thus, a Mackey functor $\uM$ gives a Mackey functor $\uM^\ast$ by flipping the arrows. The $G$-module $\Z$ with trivial action gives a Mackey functor $\uZ$ (the constant Mackey functor with $\Z$-coefficients), and by flipping arrows we get a Mackey functor $\uZ^\ast$. More explicitly, the Mackey functor $\uZ$ is given by 
$$\uZ(G/H)=\Z, ~\res^K_H = id,~tr^K_H=[K:H],$$
and $\uZ^\ast$ is given by 
$$\uZ(G/H)=\Z, ~\res^K_H = [K:H],~tr^K_H=id.$$
\end{mysubsection}

\begin{mysubsection}{$RO(G)$-graded cohomology} It turns out that the category with finite $G$-sets as objects, and homotopy classes of spectrum maps as morphisms is naturally isomorphic to the Burnside category. Therefore, the homotopy groups of $G$-spectra are naturally Mackey functors. Lewis, May and McClure \cite{LMM81} construct $RO(G)$-graded cohomology theories associated to Mackey functors, which on the objects $G/H$ gave the underlying Mackey functor. In fact, one has 
\begin{thm}\cite[Theorem 5.3]{GM95a} 
 For a Mackey functor $\uM$, there is an Eilenberg-MacLane $G$-spectrum $H\uM$ which is unique up to isomorphism in the equivariant stable homotopy category. For Mackey functors $\uM$ and $\uM'$, $\{H\uM,H\uM'\}^G \cong \Hom_{\MM_G}(\uM,\uM')$.
\end{thm}
For a $G$-Mackey functor $\uM$, we fix a model of the Eilenberg-MacLane spectrum $H\uM$ which is an $\Omega$-$G$-spectrum  and fibrant in the model structure on $\fI_G\fS$. This allows us to compute homotopy classes of maps to it's fixed point spectra via adjunctions. 

\begin{notation}
We use the notation $\tHal(X;\uM)$, $\alpha \in RO(G)$, for the reduced cohomology of a based $G$-space $X$ for the cohomology theory represented by $H\uM$. That is, 
$$ \tHal(X;\uM)\cong \{ X, S^\alpha \wedge H \uM\}^G.$$
\end{notation}

\vspace*{0.5cm}

The change of groups functors for $H\leq G$ induce an isomorphism for cohomology with Mackey functor coefficients 
$$\tilde{H}^\alpha_G(G_+\wedge_H X ; \uM)\cong \tilde{H}^\alpha_H(X; i^\ast(\uM)).$$
We will later use the notation $\downarrow^G_H$ for the functor $i^\ast$ on Mackey functors. For the Mackey functors $\uA$ and $\uZ$, we have $i^\ast(\uA)=\uA$ and $i^\ast(\uZ)=\uZ$.  Therefore, we have
\begin{myeq}\label{rest}
{H}^\alpha_G(G/H_+\wedge X ; \uA)\cong \tilde{H}^\alpha_H(X; \uA).
\end{myeq}
In particular, for $H=e$, we have $\tilde{H}^\alpha(G_+\wedge X;\uA) \cong \tilde{H}^{|\alpha|}(X; \Z)$. 

The $RO(G)$-graded theories may be further equipped with a Mackey functor-valued structure as in the definition below.   
\begin{defn}
Let $X$ be a based $G$-space, $\uM$ a Mackey functor, and $\alpha \in RO(G)$. The Mackey functor valued cohomology $\uH^{\alpha}_{G}(X;\uM)$ is defined as 
$$\uH^{\alpha}_{G}(X;\uM)(G/K) := \tilde{H}^{\alpha}_{G}({G/K}_+ \wedge X;\uM).$$
The conjugation, restriction, and transfer maps are induced by the appropriate maps of $G$-spectra. 
\end{defn}

\begin{exam}\label{Z*xi}
Let us compute the Mackey functor valued homotopy groups $\upi_k^G(S^\xi \wedge H\uZ)$.  We have cofibre sequences 
$$S(\xi)_+\to S^0 \to S^{\xi}, ~ G/e_+\to G/e_+ \to S(\xi)_+.$$
 From the second cofibre sequence we compute 
$$\upi_k(S(\xi)_+ \wedge H\uZ) \cong \begin{cases} \uZ &\mbox{ if } k=0\\
                                                       \uZ^\ast &\mbox{ if } k=1\\
                 0 &\mbox{ otherwise.}\end{cases}$$
Now applying the first cofibre sequence, we obtain 
$$\upi_k(S^\xi \wedge H\uZ) \cong \begin{cases} 
                                                       \uZ^\ast &\mbox{ if } k=2\\
                 0 &\mbox{ otherwise.}\end{cases}$$
It follows that $S^\xi \wedge H\uZ \simeq \Sigma^2 H\uZ^\ast$. 
\end{exam}

\end{mysubsection}

\begin{mysubsection}{Equivariant Anderson duality} \label{eqand} 
There is a generalization of the universal coefficient theorem 
$$0 \to \Ext(H_{m-1}(X),\Z) \to H^m(X;\Z) \to \Hom(H_m(X),\Z)\to 0$$ 
to Bredon cohomology using Anderson duality \cite{And}. Equivariant Anderson duality was studied for the group $C_2$ in \cite{Ric16}, and formulated for general finite groups in \cite{HM17}. We briefly recall the results here. Let $I^G_\Q$ (respectively $I^G_{\Q/\Z}$) be the $G$-spectrum representing the cohomology theory $\Hom(\pi_\ast^G(-), \Q)$ (respectively $\Hom(\pi_\ast^G(-),\Q/\Z)$) in the category of $G$-spectra (using the Brown representability theorem \cite[Chapter XIII.3.3]{May96}). Let $I_\Z^G$ be the homotopy fibre of $I^G_\Q\to I^G_{\Q/\Z}$. For a $G$-spectrum $E$, we write $I^G_\Z(E)$ for the function spectrum $F^G(E,I^G_\Z)$. We have the following result from the cofibre sequence of spectra 
\begin{myeq} \label{and-cof} 
I^G_\Z(E)\to F^G(E, I^G_\Q) \to F^G(E,I^G_{\Q/\Z}).
\end{myeq}  
\begin{thm}
For every $G$-spectrum $E$ and grading in $RO(G)$, there is a short exact sequence for $G$-spectra $X$
$$0 \to \Ext_L(E_{\ast -1}(X),\Z) \to I^G_\Z(E)^\ast(X) \to \Hom_L(E_\ast(X),\Z) \to 0. $$ 
\end{thm} 
The notations $\Ext_L$ and $\Hom_L$ denotes the point-wise $\Ext$ and $\Hom$ for the Mackey functor valued objects. We specialize to the case $E=H\uZ$. From the definition of $I^G_\Q$, (respectively $I^G_{\Q/\Z}$) we have $\pi^G_\ast F^G(E,I^G_\Q) \cong \Hom(\pi^G_\ast(E),\Q)$ (respectively $\pi^G_\ast F^G(E,I^G_{\Q/\Z}) \cong \Hom(\pi^G_\ast(E),\Q/\Z)$). Therefore, we observe that $F^G(H\uZ, I^G_\Q) \simeq H\uQ^\ast$,  $F^G(H\uZ, I^G_{\Q/\Z})\simeq H\uQZ^\ast$, and $I^G_\Z H\uZ\simeq H\uZ^\ast$. Moreover, the cofibre sequence \eqref{and-cof} is induced by the short exact sequence of Mackey functors $0\to \uZ^\ast \to \uQ^\ast \to \uQZ^\ast \to 0$. 

Specializing to the group $C_n$, and using the fact $H\uZ^\ast \simeq \Sigma^{2-\xi} H\uZ$, the above theorem reduces to   
\begin{thm}\label{anders}
There is a short exact sequence 
$$0 \to \Ext_L(\uH^{3-\xi-\alpha}_G(S^0;\uZ), \Z) \to \uHal(S^0;\uZ) \to \Hom_L( \uH^{2-\xi-\alpha}_G(S^0;\uZ), \Z)\to 0.$$
\end{thm}
\end{mysubsection}

\section{The category of $G$-Mackey functors} \label{GMack}
We have already defined the Mackey functors $\uA$, $\uA_S$, $\uZ$ and $\uZ^\ast$. Using these we generate a number of $G$-Mackey functors, and in some cases, computations of the set of Mackey functor homomorphisms between them.  

We start with the case $C_p$. Using the notation of \cite{Lew88}, a Mackey functor $\uM$ is described by the following diagram.
$$\xymatrix@R=0.1cm{ & \uM(C_p/C_p) \ar@/_.5pc/[dd]_{\res^{C_p}_e} \\   
\uM : \\ 
& \uM(C_p/e) \ar@/_.5pc/[uu]_{tr^{C_p}_e}}$$
In terms of these diagrams, we have the following depictions for some Mackey functors.
$$\xymatrix@R=0.1cm{    & \Z \oplus \Z \ar@/_.5pc/[dd]_{[\begin{smallmatrix} 1 & p \end{smallmatrix}]}   && & & \Z  \ar@/_.5pc/[dd]_{Id}    \\  
  \uA :      &                                                                                                                              && &\uZ :   \\ 
  & \Z \ar@/_.5pc/[uu]_{\left[ \substack{0 \\ 1}\right]}  &&& & \Z \ar@/_.5pc/[uu]_{p}}$$

$$\xymatrix@R=0.1cm{   & \Z  \ar@/_.5pc/[dd]_{p} &&&  & C  \ar@/_.5pc/[dd] \\ 
\uZ^\ast :                                                     &&&& \langle C \rangle :       \\
 & \Z \ar@/_.5pc/[uu]_{Id}                    &&&   & 0 \ar@/_.5pc/[uu]      }$$ 
The right bottom diagram defines the Mackey functor $\langle C \rangle$ for any Abelian group $C$. The Mackey functor $\uA_{C_p/e}$ is described by the diagram
$$\xymatrix@R=0.1cm{ & \Z \ar@/_.5pc/[dd]_{\Delta} \\   
\uA_{C_p/e} : \\ 
& \Z^{ p} \ar@/_.5pc/[uu]_{\nabla}}$$

\begin{exam}\label{multp}
From the Yoneda lemma, we have $\Hom_{\MM_G}(\uA,\uM) \cong \uM(G/G)$. For $G=C_p$, consider the map $\uA \to \bZ$ corresponding to the element $1\in \Z$. In explicit terms, this map is 
$$\uA(C_p/C_p) = \Z\oplus \Z \stackrel{[1~0]}{\to} \Z = \bZ(C_p/C_p) $$ 
$$\uA(C_p/e) = \Z \stackrel{0}{\to} 0 = \bZ(C_p/e).$$
The kernel of this map is $\uZ^\ast$.  Consider the map $\uA$ to $\uZ$ corresponding to $1\in \Z \cong \uZ(C_p/C_p)$. This map is given by 
$$\uA(C_p/C_p) = \Z\oplus \Z \stackrel{[1 ~p]}{\to} \Z = \uZ(C_p/C_p) $$ 
$$\uA(C_p/e) = \Z \stackrel{id}{\to} \Z = \uZ(C_p/e).$$
The kernel of this map is $\bZ$. Finally, the map $\uA$ to $\uZ^\ast$ corresponding to $1\in \Z \cong \uZ^\ast(C_p/C_p)$ is described by 
$$\uA(C_p/C_p) = \Z\oplus \Z \stackrel{[1 ~p]}{\to} \Z = \uZ^\ast(C_p/C_p)$$ 
$$\uA(C_p/e) = \Z \stackrel{p}{\to} \Z = \uZ^\ast(C_p/e).$$
Observe that the composites 
$$\bZ \to \uA \to \bZ,$$
and 
$$\uZ^\ast \to \uA \to \uZ^\ast$$ 
are multiplication by $p$. 
\end{exam}

\begin{defn}
Let $H$ be a subgroup of $G$. The restriction  
$$\downarrow^G_H : \MM_G \to \MM_H$$
and the induction
$$\uparrow^G_H: \MM_H \to \MM_G$$ 
are given by 
 $$\uparrow^G_H \uM(T) = \uM(i^*(T)),~~\downarrow^G_H\uN(S) = \uN (G\times_H S) $$ 
for each finite $G$-set T, and $H$-set $S$. The structure maps are evident as the constructions $i^\ast$ and $G\times_H -$ induce functors between the Burnside categories. 
\end{defn}
These two functors are adjoint as noted in the following proposition. 
\begin{prop}\label{epadj}\cite[Proposition 4.2]{TW90}
For $H\leq G$, the functors $\downarrow^G_H$ is a right adjoint to the functor $\uparrow^G_H$. 
\end{prop}

The group $G$ equals $C_n$ for $n$ square-free with $k$ distinct odd prime factors $p_1,\cdots, p_k$. Subgroups of $G$ correspond to divisors of $n$, which in turn are in one-to-one correspondence with subsets $I \subset \uk=\{1,\cdots,k\}$ via the correspondence 
$$ I \leftrightarrow d = \prod_{i\in I} p_i.$$
\begin{notation}
For $I\subset \uk$, write $|I|=\prod_{i\in I} p_i$, and $\# I$ for the number of elements of $I$. We write $C_I$ for the subgroup $C_{|I|}$ of $G$ ( $=C_n$), and $C_{\frac{n}{I}}$ denotes the subgroup $C_{\frac{n}{|I|}}$ of $G$. For $p~|~n$ and $p\nmid |I|$, $C_{Ip}$  denotes the subgroup $C_{|I|p}$, and $C_{\frac{n}{Ip}}$ denotes the subgroup $C_{\frac{n}{|I|p}}$. 

If $\uM$ is a Mackey functor among $\{\uA$, $\uZ$, $\uZ^\ast,\bZ\}$, we write $\uM_p$ for the Mackey functor $\uM$ defined over $C_p$, and $\uM_I$ ( respectively, $\uM_{\frac{n}{I}}$, $\uM_{Ip}$, $\uM_{\frac{n}{Ip}}$) for the corresponding Mackey functor over the group $C_I$ (respectively, $C_{\frac{n}{I}}$, $C_{I p}$, $C_{\frac{n}{I p}}$). 
\end{notation}

\vspace*{0.5cm}

In the following definition, note that for $d | n$, any transitive $G$-set may be expressed as a product of a $C_d$-set and a $C_{\frac{n}{d}}$-set. For a subgroup $H$ of $G$ write $H_1 = C_d \cap H$ and $H_2= C_{\frac{n}{d}}\cap H$, and observe that  $G/H = C_d/H_1 \times C_{\frac{n}{d}}/H_2$. The maps between two transitive $G$-sets\footnote{A $G$-map $G/H\to G/K$ exists if $H\leq K$, and is given by multiplication by $g \in G$.} also split as a product of maps between the two factors.  
\begin{defn}
Define $\boxtimes : \MM_{C_d} \times \MM_{C_{\frac{n}{d}}} \to \MM_G$ denoted by $(\uM,\uN)\mapsto \uM\boxtimes \uN$, on objects $G/H = C_d/H_1 \times C_{\frac{n}{d}}/H_2$ as 
$$\uM\boxtimes \uN (G/H) = \uM(C_d/H_1) \otimes \uN(C_{\frac{n}{d}}/H_2).$$
The conjugation maps $c_g$ are given by $c_{g_1}\otimes c_{g_2}$, where $g=(g_1,g_2)$ under the isomorphism $G\cong C_d \times C_{\frac{n}{d}}$. For an inclusion $H \leq K$, define 
$$\res^K_H= \res^{K_1}_{H_1}\otimes \res^{K_2}_{H_2}, ~~~ tr^K_H= tr^{K_1}_{H_1}\otimes tr^{K_2}_{H_2}.$$
\end{defn}

\vspace*{0.5cm}

It is easy to observe that $\uZ_d \boxtimes \uZ_{\frac{n}{d}}\cong \uZ$,   and $\uA_d \boxtimes \uA_{\frac{n}{d}}\cong \uA$. The following Mackey functors also have special importance in our computations. 
\begin{defn} \label{submacks}
For $I \subseteq \underline{k}$ and  a $C_I$-Mackey functor $\uM$, define $C_{n}$-Mackey functors
$$ \CC_{n,I}\uM: = \uM \boxtimes \uZ^*_{I^c}, \,  \KK_{n,I}\uM : = \uM \boxtimes \uZ_{I^c}.$$
If $I$ is a singleton $\{i\}$, we also use the notation $\CC_{n,p_i}$, $\CC_i$, or $\CC_{p_i}$ for $\CC_{n,\{i\}}$,  and $\KK_{n,p_i}$, $\KK_i$, or $\KK_{p_i}$ for $\KK_{n,\{i\}}$. For $d=|I|$, we also use the notation $\CC_{n,d}$ for $\CC_{n,I}$, and $\KK_{n,d}$ for $\KK_{n,I}$.  
\end{defn}

\begin{rmk}\label{Z*cn}
The computation of Example \ref{Z*xi} generalizes to 
$$H\CC_{n,d}\uM \simeq \Sigma^{2-\xi^d}\KK_{n,d}\uM$$
 for $d|n$. 
\end{rmk}

\vspace*{0.5cm}

In terms of the notation above, we also note the formula
$$\downarrow^{G}_{C_I}\CC_{{n},I} \uM \cong \downarrow^{G}_{C_I}\KK_{{n},I}\uM \cong \uM.$$

\begin{prop}\label{maps}
Let $\uM \in \MM_G$ and $\uN \in \MM_{C_I}$ for some $I \subseteq {\uk}$. Assume that $\uN$ is such that all the groups in the Mackey functor are torsion, and it has $p$-torsion only for $p \mid |I|$. We then have\\
a) A map of Mackey functors $\phi: \uM \to \KK_{n,I}(\uN)$ such that $\downarrow^G_{C_I}\phi$ is a split surjection,  is a split surjection.\\
b)  A map of Mackey functors $\psi: \CC_{n,I}(\uN) \to \uM$  such that $\downarrow^G_{C_I}\phi$ is a split injection, is a split injection.
\end{prop}
\begin{proof}
We start with the proof of part a). For $J\subset \uk$,  we have the following square 
$$\xymatrix{\uM(G/C_J) \ar[rr]^{\phi(G/C_J)} &&\KK_{n,I}(\uN)(G/C_J) \\ 
\uM(G/C_{J\cap I}) \ar[u]^{tr^{C_J}_{C_{J\cap I}}} \ar[rr]^{\phi(G/C_{J\cap I})} & &\KK_{n,I}(\uN)(G/C_{J\cap I})\ar[u]_{tr^{C_J}_{C_{J\cap I}}} }.$$ 
The transfer map $tr^{C_J}_{C_{J\cap I}}$ for the Mackey functor $\KK_{n,I}(\uN)$ equals multiplication by $\frac{|J|}{|J\cap I|}$, which is relatively prime to $|I|$. The given hypothesis on $\uN$ implies that this is an isomorphism.  Suppose $s_{I\cap J}$ is  a splitting for the map $\phi(G/C_{J\cap I})$, then, $\frac{|J\cap I|}{|J|}\circ tr^{C_J}_{C_{J\cap I}} \circ s_{I \cap J}$ is a splitting for the map $\phi(G/C_J)$. 

For b),  we use the commutative diagram for $J \subset \uk$
$$\xymatrix{ \CC_{n,I}(\uN)(G/C_I)\ar[d]_{\res^{C_J}_{C_{J\cap I }}} \ar[rr]^{\psi(G/C_J)}& &\uM(G/C_J)\ar[d]_{\res^{C_J}_{C_{J\cap I}}} \\ 
\CC_{n,I}(\uN)(G/C_{J\cap I}) \ar[rr]_-{\psi(G/C_{J\cap I})} &&\uM(G/C_{J\cap I})}$$
The hypothesis on $\uN$ implies that the left vertical restriction is an isomorphism. Suppose that $\gamma(G/C_{J\cap I})$ is the splitting for the map $\psi(G/C_{J\cap I})$, then, $\frac{|J\cap I|}{|J|} \circ \gamma(G/C_{J\cap I}) \circ \res^{C_J}_{C_{J\cap I}}$ is a splitting for the map $\psi(G/C_J)$.
\end{proof}

\vspace*{0.5cm}

For a prime $p,$ note that there is a Mackey functor inclusion $\uZ^* \to \uZ$ which is an isomorphism at $C_p/e$ and is multiplication by $p$ at $C_p/C_p.$ As a consequence, the cokernel of the map is $\langle \Z/p \rangle.$ This fact generalizes as below for inclusions $\uZ^*$ to $\uZ$ of $G$-Mackey functors.
\begin{prop}\label{exact_pn}
There is a unique map $\uZ^* \stackrel{\iota}{\to} \uZ$ such that $\iota(G/e)$ is the identity. Further, we have a short exact sequence $$0 \to \uZ^* \stackrel{\iota}{\to} \uZ \to \bigoplus\limits_{i=1}^k \KK_{n, p_i} \bZpi \to 0.$$
\end{prop}
\begin{proof}
The uniqueness of $\iota$ follows from the fact that $\iota(G/e)$ is the identity, and all the transfer maps of $\uZ^*$ are isomorphisms. As the restriction maps of $\uZ^*$ are injective, $\iota$ is injective at each orbit. Therefore, it is enough to compute the cokernel of $\iota$. We proceed by induction on $k$. For each $i$, there is a short exact sequence 
\begin{myeq}\label{exact_p}
0 \to \uZ_{p_i}^* \to \uZ_{p_i} \to  \bZpi \to 0
\end{myeq} 
This proves the case $k=1.$ By the induction hypothesis, we have the short exact sequence 
$$ 0 \to \uZ^*_{\uk\setminus \{k\}} \to \uZ_{\uk\setminus \{k\}} \to \bigoplus\limits_{i=1}^{k-1}\KK_{\frac{n}{p_i}, \{i\}}\bZpi \to 0.$$
 We apply the exact functor $\uZ^*_{p_k} \boxtimes -$ to obtain 
$$0 \to \uZ^* \to \uZ_{\frac{n}{p_k}}\boxtimes \uZ^*_{p_k} \to \bigoplus\limits_{i=1}^{k-1}\KK_{\frac{n}{p_k}, \{i\}}\bZpi \boxtimes \uZ^*_{p_k}\to 0.$$
Observe that for two distinct primes $p_i$ and $p_j$, $\uZ^*_{p_i} \boxtimes \langle \Z/p_j \rangle \cong \uZ_{p_i} \boxtimes \langle \Z/p_j \rangle$. 
As a consequence, we can identify the Mackey functors $\KK_{\frac{n}{p_i}, \{i\}}\bZpi \boxtimes \uZ^*_{p_k} \cong \KK_{n, \{i\}}\bZpi$.

We also apply the exact functor $\CC_{n, \{k\}}$ to \eqref{exact_p}, and get 
$$0 \to \uZ^\ast \to \uZ^\ast_{\frac{n}{p_k}}\boxtimes \uZ_{p_k} \to \CC_{n, \{k\}} \langle \Z/p_k \rangle \to 0 .$$ 
Using the above short exact sequences we consider the following short exact sequence of chain complexes 
$$\xymatrix{&0  \ar[d] & 0  \ar[d] & 0 \ar[d]\\
0 \ar[r]  &\uZ^* \ar[r] \ar[d]_{id}   &\uZ_{\frac{n}{p_k}}\boxtimes \uZ^*_{p_k} \ar[r]\ar[d] & \bigoplus\limits_{i=1}^{k-1} \KK_{n,\{i\}}\bZpi \ar[d] \ar[r] &0 \\
0 \ar[r] & \uZ^* \ar[r]^{\iota} \ar[d] & \uZ \ar[r] \ar[d] & \coker (\iota) \ar[d] \ar[r] &0 \\
& 0  \ar[r] & \CC_{n,\{k\}}\langle \Z/p_k \rangle \ar[r]^{\cong} \ar[d] & \CC_{n,\{k\}}\langle \Z/p_k \rangle \ar[r]\ar[d] &0 
\\ & & 0 &0}$$
As the homologies of the first and middle vertical complexes are zero, therefore the sequence 
$$0 \to \bigoplus\limits_{i=1}^{k-1} \KK_{n, \{i\}}\bZpi \to \coker(\iota) \to \CC_{n, \{k\}}\langle \Z/p_k \rangle \to 0$$  
is exact. Observe as above that $\KK_{n, p_k} \langle \Z/p_k \rangle \cong \CC_{n, p_k} \langle \Z/p_k \rangle$, as, $\gcd(p_k , \frac{n}{p_k}) =1$. Proposition  \ref{maps} implies that the sequence splits. 
\end{proof}

We make one more computation of Mackey functor maps in the proposition below. 
\begin{prop}\label{mapzI}
Let $\phi : \bZ_\II \boxtimes \uZ_{J_1}\boxtimes \uZ^\ast_{J_2} \to \bZ_{\II'} \boxtimes \KK_{p_k}\bZpk$ be a non-zero map where $k\in \II$. Then $\II'= \II\setminus \{k\}$, and up to an isomorphism of the Mackey functors, $\phi$ is the unique map obtained as $\boxtimes$ of the identity on $\bZ_{\II'}\boxtimes \uZ_{J_1}$, the map $\bZ_{p_k} \to \bZpk$, and the usual inclusion $\uZ^\ast_{J_2} \to \uZ_{J_2}$.  The kernel of $\phi$ is $p_k \cdot \bZ_\II \boxtimes \uZ_{J_1}\boxtimes \uZ^\ast_{J_2} \subset \bZ_\II \boxtimes \uZ_{J_1}\boxtimes \uZ^\ast_{J_2}$.   
\end{prop}

\begin{proof}
The description of the kernel is immediate from the description of $\phi$, so it suffices to prove that statement. Note that  $\bZ_{\II'} \boxtimes \KK_{p_k}\bZpk(G/C_d)= 0$ if $p_k \nmid d$. Assume on the contrary that $\II' \neq \II\setminus \{k\}$. If $i\neq k$ belongs to $\II$ but not $\II'$, then note that $ \bZ_\II \boxtimes \uZ_{J_1}\boxtimes \uZ^\ast_{J_2}(G/C_d)=0$ if $p_i \nmid d$. For other divisors we use the square (for $p_i \nmid d$, and $p_k \mid d$) 
$$\xymatrix@C=1.5cm{\bZ_I\boxtimes \uZ_{J_1}\boxtimes \uZ^\ast_{J_2}(G/C_d) \ar[d]^{\res^{C_{dp_i}}_{C_d}} \ar[r]^-{\phi(G/C_{dp_i})}  & \bZ_{\II'} \boxtimes \KK_{p_k}\bZpk(G/C_{dp_i})  \cong \Z/p_k \ar[d]^{\res^{C_{dp_i}}_{C_d}}_{\cong}\\
\bZ_I\boxtimes \uZ_{J_1}\boxtimes \uZ^\ast_{J_2}(G/C_d)=0 \ar[r]^-{\phi(G/C_d)}   & \bZ_{\II'} \boxtimes \KK_{p_k}\bZpk(G/C_d)\cong \Z/{p_k},}$$
to deduce that $\phi=0$ at all the $G$-orbits. If $i$ belongs to $\II'$  but not $\II$, we have $\bZ_{\II'} \boxtimes \KK_{p_k}\bZpk(G/C_d)=0$ if $p_i\nmid d$. For the other divisors we use the square (for $p_i\nmid d$) 
 $$\xymatrix@C=1.5cm{\bZ_I\boxtimes \uZ_{J_1}\boxtimes \uZ^\ast_{J_2}(G/C_d) \cong \Z \ar[r]^-{\phi(G/C_{dp_i})}  & \bZ_{\II'} \boxtimes \KK_{p_k}\bZpk(G/C_{dp_i})  \cong \Z/p_k \\
\bZ_I\boxtimes \uZ_{J_1}\boxtimes \uZ^\ast_{J_2}(G/C_d)\cong \Z \ar[u]^{tr^{C_{dp_i}}_{C_d}} \ar[r]^-{\phi(G/C_d)}   & \bZ_{\II'} \boxtimes \KK_{p_k}\bZpk(G/C_d)\cong 0 \ar[u]^{tr^{C_{dp_i}}_{C_d}},}$$ 
to deduce $\phi(G/C_{dp_i})=0$ since the transfer map is either the identity or multiplication by $p_i$ which is relatively prime to $p_k$. Hence, $\II'=\II \setminus \{k\}$. This implies either both $\bZ_I\boxtimes \uZ_{J_1}\boxtimes \uZ^\ast_{J_2}(G/C_d)$ and $\bZ_{\II'} \boxtimes \KK_{p_k}\bZpk(G/C_d)$ are $0$, or $\bZ_I\boxtimes \uZ_{J_1}\boxtimes \uZ^\ast_{J_2}(G/C_d)\cong \Z$ and   $\bZ_{\II'} \boxtimes \KK_{p_k}\bZpk(G/C_d)\cong \Z/p_k$. Among the latter values of $G/C_d$, the restriction and transfer maps for $\bZ_{\II'} \boxtimes \KK_{p_k}\bZpk$ are isomorphisms. It follows that $\phi$ is determined by its value on any of the orbits. The result now follows from this. 
\end{proof}

\section{Torsion in cohomology groups and applications for $G$-spheres}\label{cohrep}
The advantage of using $RO(G)$-graded cohomology against integer-graded cohomology of $G$-spheres, is that we have the isomorphism $\uH^\bs_G(S^V;\uM) \cong \uH^{\bs - V}_G(S^0;\uM)$. We may try to understand $\uH^\bs_G(S^V;\uM)$ in a different way through cell structures for specific $V$, and then, this gives us a relation in the $RO(G)$-graded cohomology leading to effective calculations. In our case we use $V=\xi^d$, and note the cofibre sequence 
\begin{myeq}\label{cofd2}
S(\xi^d)_+ \to S^0 \to S^{\xi^d}.
\end{myeq} 
The cohomology of $S(\xi^d)_+$ may be computed using the cofibre sequence \eqref{cofd}. This provides a relation in $\uHbs(S^0;\uM)$ which we use to prove results inductively.
 
\begin{mysubsection}{Cohomology of transitive $G$-orbits} 
The spectrum $S(\xi^d)_+$ is obtained as a cofibre of ${G/C_d}_+ \to {G/C_d}_+$, so we may compute $\uHbs(S(\xi^d)_+; 
\uM )$ from the values 
$$\tHbs({G/C_d}_+;\uM) \cong \tH^\bs_{C_d}(S^0; \downarrow^G_{C_d}\uM).$$ 
In fact, we have 
\begin{prop} \label{orbit}
Let $\uM$ be a $G$-Mackey functor and $d$ be a divisor of $n$. For all $\alpha \in RO(G)$,   
$$ \uH^{\alpha}_G({G/C_d}_+; \uM) \cong \uparrow^G_{C_d}\uH^{\alpha}_{C_d}(S^0; \downarrow^G_{C_d}\uM).$$
\end{prop}

\begin{proof}
Write $i$ for the inclusion $C_d \to G$ and  note that $i^\ast(H \uM) \simeq H(\downarrow^G_{C_d}\uM)$. Evaluating the Mackey functor $\uHal({G/C_d}_+; \uM)$ at $G/C_m$, we obtain 
\begin{align*}
\uHal({G/C_d}_+; \uM)(G/C_m) &=\{{G/C_m}_+ \wedge {G/C_d}_+, S^{\alpha} \wedge H\uM \}^G\\
                                               &\cong  \{i^\ast(G/C_m)_+), S^{\alpha}\wedge i^\ast(H\uM) \}^{C_d}  \\ 
                                               &\cong \uH^{\alpha}_{C_d}(S^0; \downarrow^G_{C_d}\uM)(i^*(G/C_m)).
\end{align*}
The result follows. 
\end{proof}

We write down explicitly the restriction and transfer maps of the Mackey functor $\uH^{\alpha}_G({G/C_d}_+ ; \uM)$ above.  These turn out to be appropriate diagonal and fold maps on the corresponding maps of the Mackey functor $\uH^{\alpha}_{C_d}(S^0)$. Let $d$ correspond the $I\subset \uk$. For $m \mid n$ which corresponds to $J\subset \uk$,  as a $C_I$-set $G/ C_J$ splits as a disjoint union of $n/|I\cup J|$-copies of $C_I/C_{I\cap J}$. Hence we write 
\begin{myeq}\label{r-equi}
G / C_J \cong C_I/ C_{I\cap J} \times G/ C_{I\cup J}
\end{myeq}
 where the action of $C_I$ on the second factor is understood to be trivial.

Now consider $J'\subset J$, and consider  $\pi^{C_J}_{C_{J^{\prime}}} : G/C_{J'} \to G/C_J$. Then, $i^\ast(\pi^{C_J}_{C_{J'}})$ using the notation of \eqref{r-equi} may be expressed as
$$ \coprod\limits_{\frac{n}{I\cup J}} \coprod\limits_{\frac{I\cup J}{I\cup J'}} C_I/C_{I\cap J' } \to \coprod\limits_{\frac{n}{I\cup J}} C_I/C_{I\cap J},$$ 
given by the canonical projection $\pi^{C_{I\cap J}}_{C_{I\cap J'}}$ on each factor. Observe that the restriction map $\res^{C_J}_{C_{J'}}$ is given by the following composite
$$\xymatrix{\bigoplus_{\frac{n}{I\cup J}} \tH^{\alpha}_{C_I}({C_I/C_{I\cap J}}_+) \ar[drrr]_{\res^{C_J}_{C_{J'}}} \ar[rrr]^{\bigoplus_{\frac{n}{I\cap J}} \res^{C_{I\cup J}}_{C_{I\cap J'}}} & && \bigoplus_{\frac{n}{I\cup J}} \tH^{\alpha}_{C_I}({C_I/C_{I\cap J'}}_+) \ar[d]^{\text{diagonal}} \\ 
& && \bigoplus_{\frac{n}{I\cup J'}} \tH^{\alpha}_{C_I}({C_I/C_{I\cap J'}}_+)}.$$
Analogously, the transfer map $tr^{C_J}_{C_{J^{\prime}}}$ for the Mackey functor $\uHal({G/C_I}_+)$ is obtained by the following commutative diagram (the fold map $A^k \to A$ is given by the sum)
$$\xymatrix{\bigoplus_{\frac{n}{I\cup J'}} \tH^{\alpha}_{C_I}({C_I/C_{I\cap J'}}_+)  \ar[drrr]_{tr^{C_J}_{C_{J'}}} \ar[rrr]^{\bigoplus_{\frac{n}{I\cup J'}} tr^{C_{I\cap J}}_{C_{I\cap J'}}} & &&\bigoplus_{\frac{n}{I\cup J'}} \tH^{\alpha}_{C_I}({C_I/C_{I\cap J}}_+) \ar[d]^{\text{fold}} \\
 & && \bigoplus_{\frac{n}{I\cup J}} \tH^{\alpha}_{C_I}({C_I/C_{I\cap J}}_+)}.$$ 

 \end{mysubsection}

 \begin{mysubsection}{Cohomology of representation spheres}
The real irreducible $G$-representations are given by the restrictions of $\xi^d$ for $d\leq n$, where $\xi^d$ and $\xi^{n-d}$ have the same underlying real representation. We compute the cohomology of $S(\xi^d)$ for $d | n$\footnote{The computations for other $d$ are similar. The cofibre sequence for $S(\xi^{dk})$ is obtained by replacing $\rho$ by $\rho^k$ if $k$ is relatively prime to $n$.}, via the following cofibre sequence of $G$-spectra \eqref{cofd} 
$${G/C_d}_+ \stackrel{\rho-1}{\to} {G/C_d}_+ \to S(\xi^d)_+.$$
As a consequence, one obtains the cohomology long exact sequence (with coefficients in any Mackey functor $\uM$)
$$\cdots \to  \uH^{\alpha -1}_G({G/C_d}_+) \to \uHal(S(\xi^d)_+) \to \uHal({G/C_d}_+) \stackrel{1-\uHal(\rho)}{\to} \uHal({G/C_d}_+) \to \cdots$$
which induces the short exact sequences
\begin{myeq}\label{ses}
0 \to \coker (1-\uH_G^{\alpha -1}(\rho)) \to \uHal(S(\xi^r)_+) \to \ker(1-\uHal(\rho)) \to 0,
 \end{myeq} 
for every $\alpha \in RO(G)$.

We now write down the structure of the Mackey functors $\coker (1-\uHal(\rho))$ and $\ker (1-\uHal(\rho))$. The following diagram of stable homotopy classes allows us to compute $\uHal(\rho)$ at the orbit $G/C_m$. In the following we assume that $d$ corresponds to $I\subset \uk$ and $m$ to $J\subset \uk$. We use the notation $i$ for the inclusion $C_d \to G$, and the bijection of $C_I$-sets $i^\ast(G/C_J) \cong \sqcup_{\frac{n}{I\cup J}} C_I/C_{I\cap J}$. 
    $$\xymatrix{\{{G/C_I}_+ \wedge{G/C_J}_+, S^{\alpha} \wedge H\uM\}^G \ar[rr]^{\uHal(\rho)(G/C_J)}\ar[d]_{i^*}^{\cong}    && \{{G/C_I}_+ \wedge {G/C_J}_+, S^{\alpha}\wedge H\uM\}^G \ar[d]^{i^*}_{\cong} \\
\{{G/C_J}_+, i^\ast S^\alpha \wedge H\uM \}^{C_I} \ar[rr]^{\widehat{\uHal(\rho)}(G/C_J)} \ar[d]_{\cong} & &\{{G/C_J}_+, i^\ast S^\alpha\wedge H\uM\}^{C_I} \ar[d]^{\cong} \\ 
\bigoplus \limits_{\frac{n}{I \cup J}}\{{C_I/C_{I\cap J}}_+, \sum^{i^\ast \alpha)}H\downarrow^G_{C_I}\uM\}^{C_I} \ar[rr]^{\widehat{\widehat{\uHal(\rho)}}(G/C_J)} &&\bigoplus \limits_{\frac{n}{I\cup J}} \{{C_I/C_{I\cap J}}_+, \sum^{i^\ast \alpha}H\downarrow^G_{C_I}\uM\}^{C_I}}$$
 In the above, the isomorphisms $i^*$ are given by \eqref{rest}, and the bottom vertical isomorphisms come from  \eqref{r-equi}. In the top row, the map represents the action of $\rho$ on cohomology. We now make an assumption on the coefficients $\uM$ that for a proper subgroup $H$, the Mackey functor values for $\uH^\alpha_H(S^0;\downarrow^G_H\uM)$ satisfy that the homomorphisms $c_h$ are trivial for $h\in H$. If we are trying to compute the cohomology with $\uA$ or $\uZ$-coefficients inductively, we may make this assumption on subgroups and prove it for $G$\footnote{The reader may note that the inductive proof of this fact is achieved through the final computation of $\uHal(S^0)$ in the later sections.}. We write $\uH^\alpha_{C_I}(S^0;\uM)(C_I/C_{I\cap J})=A$. Under this assumption, the map in the bottom row permutes the various factors in the direct sum cyclically. Let $\sigma_r$ denote the standard $r$-cycle $(1,2, \cdots, r)$ of the permutation group $\Sigma_r$. Therefore, we have 
$$ \widehat{\widehat{\uHal(\rho)}}(G/C_J) = \sigma_{\frac{n}{I\cup J}},$$
 after possibly changing the coordinates in the direct sum. Note that $\ker (1- \sigma_r: A^r \to A^r) \cong A$, the diagonal copy of $A$, and $\coker(1 - \sigma_r : A^r \to A^r) \cong A$, generated by the image of $(0 , \cdots, A , \cdots,0)$. Using this description, we deduce \\
1) For $J'\subset J\subset \uk$, the transfer map $tr^{C_J}_{C_{J'}}$ of $\ker (1-\uH_G^{\alpha -1}(\rho))$ is given by $|I\cup J|/|I\cup J'|$ times the transfer map $tr^{C_{I\cap J}}_{C_{I\cap J'}}$ of the Mackey functor $\uH^{\alpha}_{C_I}(S^0)$.  The restriction map $\res^{C_J}_{C_{J'}}$ of $\ker (1-\uH_G^{\alpha -1}(\rho))$ equals the restriction map $\res^{C_{J\cap I}}_{C_{J'\cap I}}$ of the Mackey functor $\uH^{\alpha}_{C_I}(S^0)$.\\
2) For $J'\subset J\subset \uk$, the restriction map $\res^{C_J}_{C_{J'}}$ of $\coker (1-\uH_G^{\alpha -1}(\rho))$ is given by $|I\cup J|/|I\cup J'|$ times the restriction map $\res^{C_{I\cap J}}_{C_{I\cap J'}}$ of the Mackey functor $\uH^{\alpha}_{C_I}(S^0)$.  The transfer map $tr^{C_J}_{C_{J'}}$ of $\ker (1-\uH_G^{\alpha -1}(\rho))$ equals the transfer map $tr^{C_{J\cap I}}_{C_{J'\cap I}}$ of the Mackey functor $\uH^{\alpha}_{C_I}(S^0)$.

This implies

\begin{prop} \label{coker}
With respect to the above notation, \\
1) $\ker (1-\uHal(\rho)) = \KK_{n,I}\uH^{\alpha}_{C_I}(S^0; \uM),$\\
2) $\textup{coker} (1-\uHal(\rho)) = \CC_{n,I}\uH^{\alpha}_{C_I}(S^0; \uM)$.
\end{prop}

Incorporating the results of Proposition \ref{coker} in \eqref{ses} we obtain the short exact sequence 
\begin{myeq}\label{exsphere}
0 \to \CC_{n,d}\uH^{\alpha-1}_{C_d}(S^0; \uM) \to \uHal(S(\xi^d)_+; \uM) \to \KK_{n,d}\uH^{\alpha}_{C_d}(S^0; \uM) \to 0.
\end{myeq}
\end{mysubsection}

\begin{mysubsection}{Torsions in $\uHal(S^0)$} We now use \eqref{exsphere} to deduce some general results about $\uHal(S^0)$. We assume throughout this section that the coefficient Mackey functor is either $\uA$ or $\uZ$, and drop the coefficients from the notation in this case. 
\begin{prop}\label{all}
a) Among the subset of $\alpha \in RO(G)$ such that  $|\alpha^H| >0$ for all $H \neq G$, the Mackey functor $\uHal(S^0)$ depends only on the value of $|\alpha^G|$. \\
b)  Among the subset of $\alpha \in RO(G)$ such that $|\alpha^H|<0$ for all $H \neq G$, the Mackey functor $\uHal(S^0)$ depends only on the value of $|\alpha^G|.$
\end{prop}

\begin{proof}
We prove $a)$ by induction on the number of prime factors of $n.$ A similar argument works for $b).$ For $n=p$, a prime, this is a consequence of the computations in \cite{Lew88}, and for $n =pq$ a product of two distinct primes, this appears in the proof of Proposition $5.1$ of \cite{BG19}. We apply the same technique to complete the proof in the general case.

  Assume $\alpha_1$, $\alpha_2 \in RO(G)$ such that  $|\alpha_i^H| > 0$ for $H \neq G$ and, $|\alpha_1^G| = |\alpha_2^G|$. Then for some integers $j_i$ and $l_s$, 
$$\alpha_1 + \sum_{i} \xi^{j_i} = \alpha_2 + \sum_{s} \xi^{l_s}.$$ 
 For $d \mid n,$ we have the following long exact sequence associated to \eqref{cofd2},
\begin{myeq} \label{exr}
\cdots \uH^{\alpha + \xi^d -1}_G(S(\xi^d)_+) \to \uHal(S^0) \to \uH^{\alpha +\xi^d}_G(S^0) \to \uH^{\alpha +\xi^d}_G(S(\xi^d)_+) \to \cdots 
\end{myeq}
Using the short exact sequence \eqref{ses}, and the induction hypothesis, we obtain that $\uH^{\alpha +\xi^d}_G(S(\xi^d)_+)$ and $\uH^{\alpha +\xi^d-1}_G(S(\xi^d)_+)$ are zero. Therefore, $\uHal(S^0) \cong \uH^{\alpha +\xi^d}_G(S^0)$. Thus, we deduce 
$$\uH^{\alpha_1}_G(S^0) \cong \uH^{\alpha_1 + \sum_{i} \xi^{j_i}}_G(S^0) \cong \uH^{\alpha_2 + \sum_{k} \xi^{l_k}}_G(S^0)\cong \uH^{\alpha_2 }_G(S^0).$$
Hence the result follows.
\end{proof}

From Proposition \ref{all}, and the vanishing of $\uH^{m}_G(S^0)$ for $m \neq 0$, we deduce
\begin{cor}\label{sign}
Let $\alpha \in RO(G)$. Suppose $|\alpha^H|>0$ or $|\alpha^H|<0$ for all $H\leq G$. Then $\uHal(S^0) =0$.
\end{cor}

We have the injective map of  Mackey functors $\CC_{n,d}(\uH^{\alpha-1}_{C_d}(S^0)) \to \uHal(S(\xi^d)_+)$ \eqref{exsphere}. From the cohomology exact sequence \eqref{exr}, we have the boundary map $\delta : \uH^{\alpha +\xi^d -1}_G(S(\xi^d)_+) \to \uHal(S^0)$. We now prove that the composite (in appropriate degrees) is injective. 
\begin{prop}\label{sus-r}
The composite 
$$ \CC_{n,d}(\uH^{\alpha+\xi^d-2}_{C_d}(S^0)) \to \uH^{\alpha +\xi^d-1}_G(S(\xi^d)_+) \stackrel{\delta}{\to} \uHal(S^0)$$
 is an isomorphism at $G/K$ for all $K \leq C_d$. It follows that the composite is injective at all $G/H$. 
\end{prop}

\begin{proof}
The second statement is clear from the fact that the restriction $\res^{C_m}_{C_{\gcd(m,d)}}$ is injective (in fact, $=$ identity) for the Mackey functor $\CC_{n,d}\uM$. Let $i$ denote the inclusion $C_d \to G$. The $C_d$-space $i^\ast S(\xi^d)_+$ is $S^1_+$ with trivial action. The cofibre sequence \eqref{cofd2} restricted to $C_d$ becomes the cofibre $S^1_+\stackrel{\pi}{\to} S^0 \to S^2$. As $\pi$ is a retraction, this induces an isomorphism $\uHal(S(\xi^d)_+)(G/K)\cong \tilde{H}^\alpha_{C_d}(S^0) \oplus \tilde{H}_{C_d}^{\alpha-1}(S^0)$ for $K \leq C_d$. We note that $G/K_+\simeq {G/C_d}_+\wedge C_d/K_+$, so that $\uHal(S^0)(G/K)\cong \uH^{\alpha}_{C_d}(S^0)(C_d/K)$.  It follows that the composite in the proposition is an isomorphism at $G/K$ for all $K \leq C_d$. 
\end{proof}

We now use the computational strategy to begin with an $\alpha$ as in Corollary \ref{sign}, and compute for general $\alpha$ by adding and subtracting multiples of $\xi^i$ via the exact sequence \eqref{exr}. We use this technique to prove that the odd degree elements  in $\uHal(S^0)$ are all torsion.
\begin{thm}\label{torsion}
If $\alpha$ is an odd element in $RO(G),$ then $\uHal(S^0) \otimes \Q =0.$
\end{thm}

\begin{proof}
We proceed by induction on $n$. We know that if all the fixed points of $\alpha$ are odd negative then $\uHal(S^0) =0$. Start with an $\alpha$ with $|\alpha|$ odd, and $\uHal(S^0) \otimes \Q =0$.  We show that for every $d$, $\uH^{\alpha+\xi^d}_G(S^0)\otimes \Q =0$. This will imply the result. From \eqref{exr}, we have the exact sequence 
$$0 \to \uH^{\alpha+\xi^d}_G(S^0)\otimes \Q  \to \uH^{\alpha +\xi^d}_G(S(\xi^d)_+)\otimes \Q  \stackrel{\delta}{\to} \uH^{\alpha+1}_{G}(S^0)\otimes \Q  \cdots.$$
We have another short exact sequence,
$$0 \to \CC_{n,d}\uH^{\alpha+1}_{C_d}(S^0)\otimes \Q \to \uH^{\alpha+\xi^d}_G(S(\xi^d)_+)\otimes \Q  \to \KK_{n,d}\uH^{\alpha+2}_{C_d}(S^0)\otimes \Q  \to 0$$
Using the induction hypothesis, we get  $\KK_{n,d}\uH^{\alpha+2}_{C_d}(S^0) \otimes \Q =0$. Thus, from the above short exact sequence we obtain 
$$\uH^{\alpha+\xi^d}_G(S(\xi^d)_+) \otimes \Q \cong \CC_{n,d}(\uH^{\alpha+1}_{C_d}(S^0)) \otimes \Q.$$ 
It follows then  from Proposition \ref{sus-r} that $\delta$ is injective. 
\end{proof}

Theorem \ref{torsion} helps us to deduce that the short exact sequence \eqref{exsphere} splits. 
\begin{thm}\label{splitting}
For all $\alpha \in RO(G)$, the Mackey functor 
$$\uHal(S(\xi^d)_+) \cong \CC_{n,d}(\uH^{\alpha-1}_{C_d}(S^0)) \oplus \KK_{n,d}(\uH^{\alpha}_{C_d}(S^0)).$$ 
\end{thm}

\begin{proof}
We divide the  proof into two cases: $\alpha$ is odd and even. Let us consider the short exact sequence,
$$0 \to \CC_{n,d}(\uH^{\alpha-1}_{C_d}(S^0)) \to \uHal(S(\xi^d)_+) \to \KK_{n,d}(\uH^{\alpha}_{C_d}(S^0)) \to 0.$$
First consider the case $\alpha$ odd. Theorem \ref{torsion} implies that $\KK_{n,d}\uH^{\alpha}_{C_d}(S^0)$ has only torsion. 
For $K \leq C_d,$   
$$\uHal(S(\xi^d)_+)(G/K) \cong \tilde{H}^{\alpha}_{K}(S^1_+) \cong \tilde{H}^{\alpha}_{K}(S^0) 
\oplus \tilde{H}^{\alpha-1}_{K}(S^0).$$
 So the above short exact sequence splits when evaluated at the orbit $G/K$ for all $K \leq C_d$. Thus the part a) of Proposition \ref{maps} implies the result.

For $\alpha$ even, the Mackey functor $\CC_{n,d}\uH^{\alpha-1}_{C_d}(S^0) \otimes \Q$  vanishes due to Theorem \ref{torsion}. Then, we use part b) of  Proposition \ref{maps} to conclude the result.
\end{proof}
Finally, we show that $\uHbs(S^0)$ has no torsion relatively prime to $n$. 
\begin{thm}\label{ptorsion}
The Mackey functor $\uHbs(S^0)$ has no $p$-torsion if $p$ does not divide $n.$
\end{thm}

\begin{proof} 
Write $n = p_1 \cdots p_k$, and proceed by induction on $k$. For $k =1$, this follows from the computations in \cite{Lew88}. Let $d$ be a divisor of $  n$. As the number of prime factors of $d$ is smaller than that of $n$, by induction hypothesis we have that $\CC_{n,d}\uH^{\alpha-1}_{C_d}(S^0)$ and $\KK_{n,d}\uH^{\alpha}_{C_d}(S^0; \uA)$ have no $p$-torsion. It follows from Theorem \ref{splitting} that $\uHal({S(\xi^d)}_+)$ has no $p$-torsion for every $\alpha$.

We know that if all fixed points of $\alpha$ are negative then $\uHal(S^0)=0$ (Corollary \ref{sign}). We now show that the process of adding a copy of $\xi^d$ to $\alpha$ does not pick up any $p$-torsion.  This will allow us to conclude the result. Consider the following exact sequence
$$\cdots \to \uH^{\bs-1}_G(S(\xi^d)_+) \stackrel{\delta}{\to} \uH^{\bs -\xi^d}_G(S^0) \to \uHbs(S^0) \to \uHbs(S(\xi^d)_+) \to \cdots .$$
As $\uHbs(S(\xi^r)_+)$ is $p$-torsion free, it is enough to check that the cokernel of $\delta$ does not contain any $p$-torsion.   Using the induction hypothesis, it suffices to prove that the cokernel of $\delta(G/G)$ has no $p$-torsion. For $\alpha$ odd, $\uHal(S^0)_{(p)} = 0$ by the assumption of the conclusion for $\alpha$ and Theorem \ref{torsion}. Hence, the statement follows in odd degrees. In the rest of the proof we assume that $\alpha$ is even.  In this case, we localize at $p$ to compute the following map  
$$\uH^{\alpha+\xi^d-1}_G(S(\xi^d)_+)_{(p)} \stackrel{\delta}{\to} \uHal(S^0)_{(p)} .$$
Applying Theorem \ref{splitting}, the left side is $\CC_{n,d}\uH^{\alpha + \xi^d -2}(S^0)_{(p)} \oplus \KK_{n,d}\uH^{\alpha + \xi^d -1}(S^0)_{(p)}$. The second term in the summand is $0$ by Theorem \ref{torsion}. It follows that $\delta$ is an isomorphism at $G/K$ for $K\leq C_d$ by Proposition \ref{sus-r}. This in turn implies that the action of the $c_g$ on  $\uHal(S^0)_{(p)}(G/K)$ is trivial for $K\leq C_d$. From the double coset formula, we conclude 
$$\res^G_{C_d} \circ tr^G_{C_d}(x) = \frac{n}{d}x.$$
Consider the commutative diagram
$$\xymatrix{\uH^{\alpha + \xi^d-1}_G(S(\xi^d)_+)_{(p)}(G/G) \ar@/_1pc/[d]_{\res^{G}_{C_d}} \ar[rr]^-{\delta(G/G)} &&\uHal(S^0)_{(p)}(G/G) \ar@/_1pc/[d]_{\res^{G}_{C_d}} \\ 
\uH^{\alpha + \xi^d-1}_G(S(\xi^d)_+)_{(p)}(G/C_d)\ar@/_1pc/[u]_{tr^{G}_{C_d}} \ar[rr]_-{\delta(G/C_d)} &&\uHal(S^0)_{(p)}(G/C_d)\ar@/_1pc/[u]_{tr^{G}_{C_d}}}$$
As $\frac{n}{d}$ is relatively prime to $p$, the right vertical restriction map $\res^{G}_{C_d}$ is surjective, and $\frac{d}{n}tr^{G}_{C_d}$ is a section. As a consequence, there is some Abelian group $C (= \ker(\res^G_{C_d}))$ such that 
$$\uHal(S^0)_{(p)}(G/G) = \uH^{\alpha}_G(S^0)_{(p)}(G/C_d) \oplus C.$$
As $\uHal(S^0)_{(p)}(G/G)$ does not contain any $p$-torsion, it turns out that the Abelian group $C$ has no $p$-torsion. Now,  the restriction map $\res^{G}_{C_d}$, and the transfer map $tr^{G}_{C_d}$ for the Mackey functor $\uH^{\alpha + \xi^d -1}_G(S(\xi^d)_+)_{(p)}$ are isomorphisms, and subsequently, 
$$\coker(\delta(G/G)) = \coker(\delta(G/C_d)) \oplus C.$$ 
Proposition \ref{sus-r} implies that the map $\delta(G/C_d)$ is an isomorphism. Therefore, $\coker(\delta)$ has no $p$-torsion.
\end{proof}
\end{mysubsection}

\section{Cohomology with $\uZ$ coefficients}\label{cohZ}
In this section, we discuss the computations of $\uHbs(S^0;\uZ)$. We use the iterated tensor products  
 $$\boxtimes : \MM_{C_{p_1}} \times \cdots \times \MM_{C_{p_k}} \to \MM_G.$$
We observe that the constant $G$-Mackey functor $\uZ$ is $\boxtimes(\uZ,\cdots,\uZ)$ and the pointwise dual $\uZ^\ast = \boxtimes(\uZ^\ast,\cdots,\uZ^\ast)$. 
\begin{notation} \label{tens}
For $J = (j_1,\cdots, j_k) \in \{0,1\}^k$, we write $\uZ^J = \boxtimes(\uN_1,\cdots, \uN_k)$ where 
$$\uN_i = \begin{cases} \uZ &\mbox{  if } j_i = 0 \\ 
                                    \uZ^\ast & \mbox{ if } j_i=1 . 
                \end{cases} $$ 
Hence, we have  $\uZ^{(0,\cdots ,0)} = \uZ$, and $\uZ^{(1,\cdots,1)} = \uZ^\ast$. In fact if $m= \prod\limits_{j_i=1} p_i$, then $\uZ^J = \KK_{n,m}\uZ^\ast$. 
\end{notation}

\begin{mysubsection}
{The additive structure of $\uHbs(S^0;\uZ)$} \label{addZ}
Recall that for $G =C_p$, with $p$ an odd prime, we have the following formula for $\uHbs(S^0;\uZ)$ \cite[Appendix B]{Fer00} 
 \begin{myeq} \label{fer} 
 \uH^{\alpha}_{C_p}(S^0; \uZ) = \begin{cases} \uZ & \text{if} \; |\alpha| =0, \; |\alpha^{C_p}| \leq 0 \\ 
                                                             \uZ^* & \text{if} \; |\alpha| =0, \; |\alpha^{C_p}| > 0 \\ 
                                                          \bZp & \text{if}\;|\alpha| >0,\; |\alpha^{C_p}| \leq 0, \; \text{and} \; \alpha \; \text{even} \\ 
                                                           \bZp & \text{if} \;|\alpha|<0, \; |\alpha^{C_p}| >1,\; \text{and} \; \alpha \; \text{odd} \\
                                                            0 & \text{otherwise.} \end{cases} 
\end{myeq} 

 The Mackey functors  $\uHal(S^0; \uZ)$ are all modules over $\uZ$, so they satisfy the relation   \cite[Proposition 16.3]{TW95}  
 \begin{myeq}\label{Z-mod}
tr^{G}_{H}\circ res^{G}_H (x) = [G:H] x.
\end{myeq}
Using \eqref{fer} and \eqref{Z-mod} together we observe that $\uHal(S^0;\uZ)$ is $0$ for a range of values of $\alpha$. 
\begin{thm}\label{zercoh}
Let  $\alpha \in RO(G).$ Then, \\
1) If $|\alpha|>0$ odd,  $\uHal(S^0; \uZ) =0.$\\
2) If $|\alpha|<0$ even,  $\uHal(S^0; \uZ) =0.$\\
3) If $|\alpha|<0$ odd and $|\alpha^{C_{p_i}}| \leq 1$ for all $1 \leq i \leq k$,  $\uHal(S^0; \uZ) =0.$
\end{thm}

\begin{proof}
We use induction on $k$ to conclude the result. For $k=1,$ this follows readily from \eqref{fer}, and for $k=2$, this follows from \cite[Theorem 7.3]{BG19}. Using the induction hypothesis, we deduce that the Mackey functor $\uHal(S^0; \uZ)$ is $0$ at any $G/H$ for $H\neq G$, and hence is of the  form $\langle A \rangle$ for some Abelian group $A$. Now we note from \eqref{Z-mod} that $a\in A$ must satisfy $pa=0$ for any $p\mid n$. By taking two distinct prime factors $p$ and $q$ of $n$ we have $pa=0$ and $qa=0$ which together implies $a=0$. Hence the result follows.
\end{proof}

Theorem \ref{zercoh} provides a starting point for the computation of $\uHal(S^0; \uZ)$. We use the cofibre sequences \eqref{cofd2}
$$S(\xi)_+ \to S^0 \to S^{\xi},$$ 
$$S(\xi^{p_i})_+ \to S^0 \to S^{\xi^{p_i}}$$ 
for $1 \leq i \leq k,$ and associated long exact sequences 
\begin{myeq}\label{ex-e}
\cdots \uHalxi(S^0) \to \uHal(S^0) \to \uHal(S(\xi)_+) \to \uH^{\alpha+1-\xi}_G(S^0) \cdots
\end{myeq}
and
\begin{myeq}\label{shex-pi}
\cdots \uH^{\alpha-\xi^{p_i}}_G(S^0) \to \uHal(S^0) \to \uHal(S(\xi^{p_i})_+) \to \uH^{\alpha+1 -\xi^{p_i}}_G(S^0)\cdots.
\end{myeq}
We also recall from Proposition \ref{splitting} that for $\alpha \in RO(G)$, we have
\begin{myeq}\label{constsphere}
 \uHal(S(\xi^{p_i}); \uZ) \cong \KK_i\uH^{\alpha}_{C_{p_i}}(S^0; \uZ)\oplus \CC_i\uH^{\alpha-1}_{C_{p_i}}(S^0; \uZ).
\end{myeq}
Hence \eqref{constsphere} may be used together with \eqref{fer} to write down the cohomology of the representation sphere $S(\xi^{p_i})$. We deduce
\begin{thm}\label{non-zero}
For $\alpha \in RO(G)$ with $|\alpha|<0$ odd, 
$$\uHal(S^0; \uZ) \cong \bigoplus_{|\alpha^{C_{p_i}}|>1} \KK_i\bZpi. $$
For $|\alpha| >0$ even, 
$$\uHal(S^0; \uZ) \cong \bigoplus_{|\alpha^{C_{p_i}}|\leq 0}\KK_i\bZpi.$$
\end{thm}

\begin{proof}
We start with the case $|\alpha| < 0$ odd. Here we first observe 
 $$\uHal(S^0; \uZ) \cong \begin{cases} 
                      \uH^{\alpha -\xi^{p_i}}_G(S^0; \uZ) & \mbox{if}\; |\alpha^{C_{p_i}}| \neq 3 \\ 
                      \uH^{\alpha -\xi^{p_i}}_G(S^0; \uZ)\oplus \KK_i\bZpi & \mbox{if} \; |\alpha^{C_{p_i}}| =3 
                    \end{cases}$$ 
by working the exact sequence \eqref{shex-pi} in each case. 

If $|\alpha^{C_{p_i}}|\leq 1,$  \eqref{constsphere} implies 
$$\uH^{\alpha-1}_G(S(\xi^{p_i})_+; \uZ) =0~\mbox{and}~ \uHal(S(\xi^{p_i})_+; \uZ) =0,$$ 
and therefore,  
$$\uH^{\alpha-\xi^{p_i}}_G(S^0; \uZ) \cong \uHal(S^0; \uZ).$$
If $|\alpha^{C_{p_i}}|=3$,  \eqref{constsphere} gives  
$$\uH^{\alpha-1}_G(S(\xi^{p_i})_+; \uZ) =0, 	~ \uHal(S(\xi^{p_i})_+; \uZ) =\KK_i\bZpi.$$ 
Hence, we obtain the short exact sequence 
$$0 \to\uH^{\alpha-\xi^{p_i}}_G(S^0; \uZ) \to \uHal(S^0;\uZ) \to \KK_i\bZpi \to 0.$$ 
The map $\uHal(S^0;\uZ) \to \KK_i\bZpi$ above is split at $C_{p_i}$, because $\res_{C_{p_i}}(S(\xi^{p_i}))\simeq S^1$ with trivial action, and the map is induced by the inclusion of the base point which has a retract given by the projection $S^1_+ \to S^0$. We now apply Proposition \ref{maps} to deduce $\uHal(S^0;\uZ) \cong \uH^{\alpha -\xi^{p_i}}_G(S^0;\uZ) \oplus \KK_i\bZpi.$

If $|\alpha^{C_{p_i}}| >3$, \eqref{constsphere} implies 
$$\uH^{\alpha-1}_G(S(\xi^{p_i})_+; \uZ) =\CC_i\bZpi ~\mbox{and}~ \uHal(S(\xi^{p_i})_+;\uZ) =\KK_i\bZpi,$$ 
and Theorem \ref{zercoh} implies 
$$\uH^{\alpha - 1}_G(S^0;\uZ)=0,~~ \uH^{\alpha +1-\xi^{p_i}}_G(S^0;\uZ)=0.$$   
Inputting this data into \eqref{shex-pi} gives the exact sequence
 $$0 \to \CC_i\bZpi \stackrel{\phi_\alpha}{\to} \uH^{\alpha-\xi^{p_i}}_G(S^0;\uZ) \to \uHal(S^0;\uZ) \to \KK_i \bZpi \to 0.$$ 
It follows by restricting to $C_{p_i}$ as above that
$$\uHal(S^0;\uZ) \cong \coker(\phi_\alpha) \oplus \KK_i\bZpi $$ 
and
$$\uH^{\alpha -\xi^{p_i}}_G(S^0;\uZ) \cong \coker(\phi_\alpha) \oplus \CC_i\bZpi.$$
Finally observe that as $p_i$ and $n/p_i$ are relatively prime, $\KK_i\bZpi \cong \CC_i\bZpi$.  
This completes the proof in the odd case.

For $|\alpha|>0$ even we apply Theorem \ref{anders} to obtain the exact sequence 
$$0\to \Ext_L(\uH^{3-\xi-\alpha}_G(S^0;\uZ),\Z) \to \uHal(S^0;\uZ) \to \Hom_L(\uH^{2-\xi-\alpha}_G(S^0;\uZ), \Z) \to 0.$$
If $|\alpha|>0$ even, $|3-\xi-\alpha|<0$ odd and $|2-\xi-\alpha| <0$ even. Theorem \ref{zercoh} implies that the right group is $0$. The even case now follows from the odd one by the observation $\Ext_L(\KK_i\bZpi,\Z)\cong \KK_i\bZpi$. 
\end{proof}

 Let us consider the function 
$\JJ : RO(G) \to \{0,1\}^k$ with each component given by 
$$\JJ(\alpha)_i=\begin{cases} 0 &\mbox{if} \; |\alpha^{C_{p_i}}| \leq 0 \\ 
                                             1 & \mbox{if} \;|\alpha^{C_{p_i}}| >0 .
\end{cases}$$
This defines the Mackey functor $\uZ^{\JJ(\alpha)}$ (Notation \ref{tens}) which is the value of $\uHal(S^0;\uZ)$ for $|\alpha|=0$ by the theorem below. 
\begin{thm}\label{zero}
For $|\alpha|=0$, 
$$\uHal(S^0; \uZ) \cong \uZ^{\JJ(\alpha)}.$$
\end{thm}
\begin{proof}
 We apply Corollary \ref{zercoh} and Theorem \ref{non-zero} to the exact sequence \eqref{ex-e}. It turns out the short exact sequence 
$$0 \to \uHal(S^0;\uZ) \to \uZ \to \bigoplus_{|\alpha^{C_{p_i}}|>0} \KK_{n, p_i}\bZpi \to 0.$$ 
Let $m = \prod\limits_{|\alpha^{C_{p_i}}|>0}p_i.$ We may rewrite the short exact sequence above as 
$$0 \to \uHal(S^0; \uZ) \to \KK_{n,m}\uZ \to \bigoplus\limits_{p_i \mid m}\KK_{n,m}(\KK_{m,p_i}\bZpi)\to 0$$ 
Note that the functor $\KK_{n,m}$ is exact and additive. Thus Proposition \ref{exact_pn} implies that $\uHal(S^0; \uZ) \cong  \KK_{n,m} \uZ^\ast.$ Hence the result follows. 
\end{proof}

\begin{exam}
The formulas in Theorems \ref{zercoh}, \ref{non-zero} and \ref{zero}, yield formulas for the homology and cohomology of $G$-spheres. Let $V$ be a $G$-representation. Observe that $|V - \dim(V)|=0$ and $\JJ(V - \dim(V))= (0,\cdots,0)$. Therefore, 
$$\uH^{V-\dim(V)}_G(S^0;\uZ) \cong \uH_{\dim(V)}^G(S^V;\uZ) \cong \uZ.$$
We also have $\JJ(\dim(V)-V)=(1,\cdots, 1)$, so that  
$$\uH^{\dim(V)-V}_G(S^0;\uZ) \cong \uH^{\dim(V)}_G(S^V;\uZ) \cong \uZ^\ast.$$
Let $V=\xi+ \sum\limits_{i=1}^{k-1} \xi^{p_1\cdots p_i}$. In this case we conclude, 
$$\uH_m^G(S^V;\uZ) \cong \begin{cases} \uZ &\mbox{ if } m=2k \\ 
                            \bigoplus\limits_{i=k-r}^{k}  \KK_i\bZpi &\mbox{ if }  m =2r, ~0\leq r < k \\ 
                            0   & \mbox{ otherwise,} \end{cases}$$
and, 
$$\uH^m_G(S^V;\uZ) \cong \begin{cases} \uZ^\ast &\mbox{ if } m=2k \\ 
                            \bigoplus\limits_{i=k-r+1}^{k}  \KK_i\bZpi &\mbox{ if }  m =2r+1, ~1\leq r \leq k-1 \\ 
                            0   & \mbox{ otherwise.} \end{cases}$$                                             
                                                           
\end{exam}

\end{mysubsection}

\begin{mysubsection}{The ring structure for $\tHbs(S^0;\uZ)$}\label{ringZ} 
In this section, we compute the ring structure on the cohomology with $\uZ$ coefficients. For $G=C_p$, this is computed in \cite{Lew88}. We calculate the ring structure on $\tHbs(S^0;\uZ)= \uHbs(S^0;\uZ)(G/G)$. 
\begin{defn}\label{malph}
For $\alpha \in RO(G)$, we define $m(\alpha)$ as 
$$ m(\alpha) := \begin{cases} 
      \prod \limits_{|\alpha^{C_{p_i}}| \leq 0} p_i & \mbox{if } |\alpha| \mbox{ is even} \\ 
      \prod \limits_{|\alpha^{C_{p_i}}|>1} p_i   & \mbox{if } |\alpha| \mbox{ is odd}. \end{cases}$$ 
\end{defn}

\begin{remark}
If $\alpha$ is of the form $V-\dim(V)$ for a $G$-representation $V$, then $m(\alpha)=n$. 
\end{remark}

We summarize the values of $\tHbs(S^0;\uZ)= \uHbs(S^0;\uZ)(G/G)$ from Theorems \ref{zero}, \ref{non-zero} and \ref{zercoh}  
\begin{myeq}\label{Zval}
 \tHal(S^0; \uZ) = \begin{cases} \Z & \text{if} \; |\alpha| =0 \\ 
                                                 \Z/m(\alpha) & \text{if} \; |\alpha| >0 \mbox{ even} \\ 
                                                \Z/m(\alpha) & \text{if} \; |\alpha| <0 \mbox{ odd} \\ 
                                                            0 & \text{otherwise.} \end{cases} 
\end{myeq}

Suppose that $r$ and $n$ are relatively prime. For $d|n$, by smashing \eqref{cofd} and \eqref{cofd2} with $H\uZ$ we obtain 
$$H\uZ \wedge S^{\xi^d} \simeq H\uZ \wedge S^{\xi^{rd}}.$$ 
This implies we have a unit in $\tH_G^{\xi^{rd}-\xi^d}(S^0;\uZ)$. Therefore in the ring structure, up to units, we may restrict the gradings $\alpha\in RO(G)$ to linear combinations of  $\xi^d$ as $d$ runs over divisors of $n$. For the rest of the section we assume that the grading is of this form. 

\begin{notation}\label{pos-neg}
We denote by $\tHp(S^0;\uZ)$ (respectively $\tHm(S^0;\uZ)$, $\tHpl(S^0;\uZ)$, and $\tHz(S^0;\uZ)$) the non-negative dimensional part (respectively negative dimensional, positive dimensional, and zero dimensional part) of $\tHbs(S^0;\uZ)$ (that is, the part of $\tHal(S^0;\uZ)$ for $|\alpha| \geq 0$ (respectively $|\alpha|<0$, $|\alpha|>0$, and $|\alpha|=0$)). It readily follows that $\tHp(S^0;\uZ)$ and $\tHz(S^0;\uZ)$ are commutative rings (as they are concentrated in even degrees by \eqref{Zval}) and, $\tHm(S^0;\uZ)$ and $\tHpl(S^0;\uZ)$ are $\tHp(S^0;\uZ)$-modules.  We denote by $\tHr(S^0;\uZ)$ (respectively $\tHrz(S^0;\uZ)$) the part of $\tHbs(S^0;\uZ)$ in gradings of the form $V-m$ (respectively of the form $V-\dim(V)$) for $G$-representations $V$.  
\end{notation}

 We will first describe $\tHp(S^0;\uZ)$ using generators and relations. The generators used are defined  in \cite [Section 3]{HHR16} which we now recall. 
\begin{defn}\label{uagen}
Let $V$ be a $G$-representation. We have the $G$-map $S^0 \to S^V$ which induces 
$$S^0 \to S^V\wedge S^0 \to S^V \wedge H\uZ$$ 
which we call $a_V\in \tH_G^V(S^0;\uZ)$. If $V$ is an oriented $G$-representation, 
$$\tH_G^{V-\dim(V)}(S^0;\uZ)\cong \tH^G_{\dim(V)}(S^V;\uZ) \cong \Z$$
 and the choice of generator is defined as $u_V$.  
\end{defn}
We also have relations among these generators namely $u_V u_W = u_{V\oplus W}$ and $a_V a_W = a_{V\oplus W}$. It follows that these classes are products of $u_{\xi^d}$ and $a_{\xi^d}$ as $d<n$ varies over divisors of $n$.  

We note that $\tH_G^{V-m}(S^0;\uZ) \cong \tH_m^G(S^V;\uZ)$ which implies that the group is $0$ if $m>\dim(V)$. It follows that $\tHal(S^0;\uZ)$ for $\alpha$ of the form $V-m$ is a subring of $\tHp(S^0;\uZ)$. This has been computed for $n=pq$ in \cite{Gh19}.  We note from \eqref{Zval} that  $\tH_G^{\xi^d}(S^0;\uZ) \cong \Z/(\frac{n}{d})$ so that 
\begin{myeq} \label{areln} 
\frac{n}{d} a_{\xi^d} = 0. 
\end{myeq}

Write $RO_0(G)\subset RO(G)$ to be the elements of dimension $0$. For $\alpha\in RO_0(G)$, $\tH^\alpha_G(S^0;\uZ) \cong \Z$ \eqref{Zval} and we easily deduce from Theorem \ref{zero} that 
\begin{myeq}\label{zerres}
\res^G_e = \frac{n}{m(\alpha)} : \Z \cong \tH_G^\alpha(S^0;\uZ) \to \tH_G^\alpha({G/e}_+;\uZ) \cong \Z. 
\end{myeq}
The cohomology ring of ${G/e}_+$ may be computed from non-equivariant cohomology and we have the formula 
$$\tHz({G/e}_+;\uZ) \cong \Z[\mu_{\xi^d}^\pm \mid d| n] $$
where $\mu_{\xi^d}=\res^G_e(u_{\xi^d})$. We now deduce the ring structure of $\tHz(S^0;\uZ)$ from the fact that $\res^G_e$ is a ring map which is injective by \eqref{zerres}. We note that $\alpha\in RO_0(G)$ is uniquely determined by its collection $\{|\alpha^{C_s}| \mid s\neq 1 \}$ of fixed points dimensions.  Therefore we have the following computation of $\tHz(S^0;\uZ)$. 
\begin{theorem}\label{ringz} 
The ring $\tHz(S^0;\uZ)$ is a subring of $\Z[u_{\xi^d}^\pm \mid d|n]$ such that the degree $\alpha$ part includes in the right hand side by multiplication by $n/m(\alpha)$. In  particular $\tHrz(S^0;\uZ)$  is isomorphic to  $\Z[u_{\xi^d} \mid d|n ]$.  
\end{theorem}

We make the following notation in view of Theorem \ref{ringz}.
\begin{notation}
For $\alpha = \sum \limits_{d \mid n} c_d (\xi^d-2)$, we write $[\frac{n}{m(\alpha)} \prod \limits_{d \mid n} u_{\xi^d}^{c_d}]$ for the generator of $\tHal(S^0;\uZ)$. This restricts to $\frac{n}{m(\alpha)} \prod \limits_{d \mid n} \mu_{\xi^d}^{c_d} $ in $\tHal({G/e}_+;\uZ)$. 
\end{notation}

For $|\alpha| > 0$ even, we write 
$$\alpha =  c \xi + \sum_{d \mid n} c_d (\xi^d-2),$$ 
and we apply \eqref{ex-e} to deduce that $\tHal(S^0;\uZ)$ is $\Z/m(\alpha)$ generated by the class 
$a_\xi^c [\frac{n}{m(\alpha)} \prod \limits_{d \mid n} u_{\xi^d}^{c_d}]$. Hence, we deduce 
\begin{theorem} \label{ringp}
The ring $\tHp(S^0;\uZ)$ is generated over $\tHz(S^0;\uZ)$ by the class $a_\xi$ subject to the relations $na_\xi = 0$, and for $\alpha = \sum\limits_{d\mid n} c_d(\xi^d-2)$,
$$ m(\alpha) a_\xi [\frac{n}{m(\alpha)} \prod \limits_{d \mid n} u_{\xi^d}^{c_d}] =0.$$
\end{theorem} 
From Theorems \ref{ringp} and \ref{ringz}, we readily deduce that the multiplication by $u_{\xi^d}$ is injective on $\tHp(S^0;\uZ)$. On $\tHal(S^0;\uZ)$, this is multiplication by $\frac{m(\alpha+\xi^d - 2)}{m(\alpha)}$. We then apply this to the commutative square (for divisors $d,s$ of $n$)
$$\xymatrix{ \tH^0_G(S^0;\uZ) \ar[r]^{a_{\xi^d}} \ar[d]^{u_{\xi^d}} & \tH_G^{\xi^d}(S^0;\uZ)  \ar[d]^{u_{\xi^s}}  \\ 
   \tH^{\xi^d -2}_G(S^0;\uZ) \ar[r]^{a_{\xi^s}}  & \tH^{\xi^s+\xi^d -2}_G(S^0;\uZ) }$$
to deduce 
\begin{myeq} \label{aureln}
\frac{d}{\gcd(d,s)}a_{\xi^s}u_{\xi^d}= \frac{s}{\gcd(d,s)}u_{\xi^s}a_{\xi^d}.
\end{myeq}
In particular, we have for $k< \frac{n}{d}$, $a_{\xi^{dk}}u_{\xi^d}= k u_{\xi^{dk}}a_{\xi^d}$. For the group $C_{p^m}$, an analogue of \eqref{aureln} was proved in \cite[Theorem 3.5]{HHR17} and called the {\it gold (or au) relation}. 
\begin{exam}
For a proper divisor $d$ of $n$, consider the grading $\xi^d+\xi^{n/d}-2$. We have $m(\xi^d+\xi^{n/d}-2)=n$, so that $\tH^{\xi^d+\xi^{n/d}-2}_G(S^0;\uZ) \cong \Z/n$. Among the monomials in the $a$-classes and the $u$-classes, the only ones in this grading are $a_{\xi^d}u_{\xi^{n/d}}$ and $a_{\xi^{n/d}}u_{\xi^d}$. The relation \eqref{areln} implies that 
$$\frac{n}{d}a_{\xi^d}u_{\xi^{n/d}}=0, ~ d a_{\xi^{n/d}}u_{\xi^d}=0,$$
and \eqref{aureln} does not give any further relation. We prove in Theorem \ref{ringrep} that these generate the cyclic subgroups of order $n/d$ and $d$ respectively in the group $\Z/n$. 

Let $d_1,d_2, d_3$ be proper divisors of $n$ such that $d_1d_2d_3=n$. Consider the grading $\xi^{d_1}+\xi^{d_2}+\xi^{d_3} - 4$, and observe that $m(\xi^{d_1}+\xi^{d_2}+\xi^{d_3}-4)=n$. The monomials in this grading are $a_{\xi^{d_1}}u_{\xi^{d_2}}u_{\xi^{d_3}}$, $u_{\xi^{d_1}}a_{\xi^{d_2}}u_{\xi^{d_3}}$, and $u_{\xi^{d_1}}u_{\xi^{d_2}}a_{\xi^{d_3}}$. From \eqref{areln}, we have 
$$d_2d_3a_{\xi^{d_1}}u_{\xi^{d_2}}u_{\xi^{d_3}}=0,~d_1d_3u_{\xi^{d_1}}a_{\xi^{d_2}}u_{\xi^{d_3}}=0, ~d_1d_2u_{\xi^{d_1}}u_{\xi^{d_2}}a_{\xi^{d_3}}=0,$$
and from \eqref{aureln}, we have 
$$d_2a_{\xi^{d_1}}u_{\xi^{d_2}}u_{\xi^{d_3}}= d_1u_{\xi^{d_1}}a_{\xi^{d_2}}u_{\xi^{d_3}},$$ 
$$d_3 a_{\xi^{d_1}}u_{\xi^{d_2}}u_{\xi^{d_3}}=d_1u_{\xi^{d_1}}u_{\xi^{d_2}}a_{\xi^{d_3}},$$
$$d_3u_{\xi^{d_1}}a_{\xi^{d_2}}u_{\xi^{d_3}}= d_2 u_{\xi^{d_1}}u_{\xi^{d_2}}a_{\xi^{d_3}}.$$
We prove in Theorem \ref{ringrep} that there are no further relations in this grading. More explicitly, $a_{\xi^{d_1}}u_{\xi^{d_2}}u_{\xi^{d_3}}$ generates the cyclic subgroup (of the group $\Z/n$) of order $d_2d_3$, $u_{\xi^{d_1}}a_{\xi^{d_2}}u_{\xi^{d_3}}$ generates the cyclic subgroup of order $d_1d_3$, and $u_{\xi^{d_1}}u_{\xi^{d_2}}a_{\xi^{d_3}}$ generates the cyclic subgroup of order $d_1d_2$.
\end{exam}

The equations \eqref{areln} and \eqref{aureln}  are the only relations in degrees of the form $V-m$ for representations $V$ and we have the following result. 
\begin{theorem}\label{ringrep}
The ring $\tHr(S^0;\uZ)$ is generated by the classes $u_{\xi^d}$, $a_{\xi^d}$ for divisors $d$ of $n$ such that $d\neq n$, subject to the relations 
$$ \frac{n}{d}a_{\xi^d}=0, ~~ \frac{d}{\gcd(d,s)}a_{\xi^s}u_{\xi^d}= \frac{s}{\gcd(d,s)}u_{\xi^s}a_{\xi^d}.$$
\end{theorem} 
\begin{proof}
Let $\RR$ be the ring generated by the $u_{\xi^d}, a_{\xi^d}$ modulo the relations \eqref{areln},\eqref{aureln}. For $\alpha=\sum c_d \xi^d - 2\mu$, we see that the degree $\alpha$ part of $\RR$ is a sum of monomials of $\mu$-fold product of $u_{\xi^d}$s and a $|\alpha|/2$-fold product of $a_{\xi^d}$s.  We assume $|\alpha|>0$, so that the products lie in a cyclic group of order $m(\alpha)$.  Since multiplication by $u_{\xi^d}$ is injective, a monomial $\prod \limits_j u_{\xi^{d_j}}^{c_j} \prod \limits_l a_{\xi^{\delta_l}}^{e_l}$ generates a cyclic subgroup of order $\gcd(\frac{n}{\delta_l})$. Hence, using the gold relation \eqref{aureln}, the homogeneous part of $\RR$ in degree $\alpha$ is a cyclic group of order the least common multiple of the above $\gcd$s. If $p$ is a prime dividing all the $\frac{n}{\delta_l}$, $p$ does not divide the $\delta_l$, so that  
$$|\alpha^{C_p}| = \sum_j (|(\xi^{d_j})^{C_p}| - 2) \leq 0,$$ 
and thus, $p$ divides $m(\alpha)$. It follows that the degree $\alpha$ part of $\RR$ is cyclic of order a divisor of $m(\alpha)$. Hence, it suffices to prove that the classes $u_{\xi^d}$, $a_{\xi^d}$ indeed generate $\tHr(S^0;\uZ)$. 

We start with  $\alpha=\sum c_d \xi^d - 2\mu$, and use induction on $\mu$ to prove that $\tHal(S^0;\uZ)$ is generated by sums of monomials on $u_{\xi^d}$ and $a_{\xi^d}$. If $\mu=0$, the only possible monomial is $\prod \limits_d a_{\xi^d}^{c_d}$ which generates a cyclic group of order $\gcd(\frac{n}{d} \mid c_d>0)=m(\alpha)$. For $\mu>0$, choose a $d$ such that $c_d>0$. Then, $u_{\xi^d} \tH^{\alpha-\xi^d+2}_G(S^0;\uZ)$ is a cyclic subgroup of $\Z/m(\alpha)$ generated by $\frac{m(\alpha)}{m(\alpha-\xi^d +2)}$, and by induction hypothesis they are in the image of $\RR$. As $d$ varies over $c_d>0$, we show that these span the cyclic group $\Z/m(\alpha)$. Suppose for some $p\mid m(\alpha)$ that  $p\mid  \frac{m(\alpha)}{m(\alpha-\xi^d +2)}$ for all $d$. Then $p \nmid m(\alpha -\xi^d +2)$, so that $|\alpha^{C_p}| + 2 - |(\xi^d)^{C_p}| > 0$, which implies $p \nmid d$ for all $d$. This implies that $|\alpha^{C_p}|\leq -2$ (as $\mu \geq 1$), and thus,  $|\alpha^{C_p}| + 2 - |(\xi^d)^{C_p}| \leq 0$. Therefore, for every $p \mid m(\alpha)$, there is a $d$ such that $p\nmid  \frac{m(\alpha)}{m(\alpha-\xi^d +2)}$, so that the sum of cyclic groups equal $\Z/m(\alpha)$, and we are done by induction.
\end{proof}

Finally, we describe the $\tHp(S^0;\uZ)$-module $\tHm(S^0;\uZ)$. As elements in $\tHm(S^0;\uZ)$ multiply to $0$, this will complete the description of the ring $\tHbs(S^0;\uZ)$. From equivariant Anderson duality (Theorem \ref{anders}) we have the following short exact sequence 
\begin{myeq}\label{eqandr} 
0 \to \Ext(\tH^{3-\xi-\alpha}_G(S^0;\uZ), \Z) \to \tHal(S^0;\uZ) \to \Hom(\tH^{2-\xi - \alpha}_G(S^0;\uZ),\Z) \to 0.
\end{myeq}
For $|\alpha|<0$ odd, this gives an isomorphism 
$$\tHal(S^0;\uZ) \cong  \Ext(\tH^{3-\xi-\alpha}_G(S^0;\uZ), \Z) \cong \Hom(\tH^{3-\xi - \alpha}_G(S^0;\uZ),\Q/\Z).$$
The exact sequence \eqref{eqandr} is induced from the short exact sequence of Mackey functors 
$$0\to \uZ \to \uQ \to \uQZ \to 0,$$
 which is actually a short exact sequence of $\uZ$-modules. Therefore, on cohomology  they induce a long exact sequence of $\tHbs(S^0;\uZ)$-modules. Hence we have 
\begin{theorem} \label{ringodd}
$\tHm(S^0;\uZ)$ is isomorphic to $\Sigma^{3-\xi}\Hom(\tHpl(S^0;\uZ),\Q/\Z)$ as $\tHp(S^0;\uZ)$-modules.
\end{theorem}
This completes the computation of the ring structure of $\tHbs(S^0;\uZ)$.
\end{mysubsection}
\section{The additive structure for $\uA$-coefficients}\label{cohA}

The primary subject of this section is the Mackey functor valued cohomology with $\uA$-coefficients. Computations for $G=C_p$ were made by Stong and Lewis (\cite{Lew88}) and for $C_{pq}$ in \cite{BG19}. The method for $C_{pq}$ involved adding copies of $\xi^d$ ($d=1,$ $p$, or $q$) to negative $\alpha$ values, and subtracting $\xi^d$ ($d=1$, $p$, or $q$) from positive $\alpha$ values, via the exact sequences \eqref{exr}. We take a different approach here by proving that in many cases, $\uHal(S^0;\uA)$ decomposes into a direct sum of copies obtained from $\uHald(S^0;\uZ)$ for various divisors $d$ of $n$. We make the following definition 
\begin{defn}\label{rnz}
We say that $\alpha\in RO(G)$ is {\it non-zero} if all the fixed point dimensions $|\alpha^{C_d}|$ are non-zero.  We say that $\alpha$ is {\it mostly non-zero} if $|\alpha^{C_d}|=0$ implies $|\alpha^{C_{dp}}|\neq 0$ for all $p$ dividing $n$ but not $d$. We say that $\alpha$ has {\it many zeros} if it is not mostly non-zero. 
\end{defn}

\begin{exam}
Let $n=pq$ so that $G=C_{pq}$. Consider $\alpha= \xi^p + \xi^q -2$. The fixed point dimensions are given by
$$|\alpha| = 2, ~ |\alpha^{C_p}|=|\alpha^{C_q}|=0,~ |\alpha^{C_{pq}}|=-2.$$ 
Thus $\alpha$ is not non-zero but it is mostly non-zero. 
\end{exam}
For mostly non-zero $\alpha$, we will write down a formula for $\uHal(S^0;\uA)$. This will prove an {\it ``independence result''} : For $\alpha$ mostly non-zero, the Mackey functor $\uHal(S^0;\uA)$ is determined up to isomorphism by the fixed point dimensions of $\alpha$.

\begin{notation}
For the rest of the section, for $I\subset \uk$, $J$ denotes the complement $\uk \setminus I$. 
\end{notation}

 We now describe the strategy used in the calculations. Note that for each prime $p$, we have the short exact sequences 
\begin{myeq}\label{zses}
0 \to \bZ_p \to \uA_p \to \uZ_p \to 0
\end{myeq}
 and 
\begin{myeq}\label{z*ses}
0 \to \uZ^*_p \to \uA_p \to \bZ_p \to 0
\end{myeq} 
such that the composite $\bZ_p \to \uA_p \to \bZ_p$ of the maps from \eqref{zses} and \eqref{z*ses} is given by multiplication by $p$. Let $j \in I $. From \eqref{zses} and \eqref{z*ses}, we have the following short exact sequences of $G$-Mackey functors (see Definition \ref{submacks}) 
\begin{myeq}\label{tensorex1}
0 \to\bZ_{p_j}\otimes \uA_{I \setminus \{ j\}} \otimes \uZ_J \to \uA_I \otimes \uZ_J \to \uA_{I \setminus \{ j\}}\otimes \uZ_{J \cup \{j\}} \to 0,
\end{myeq} 
and 
\begin{myeq}\label{tensorex2}
0 \to \uZ^*_{p_j}\otimes \uA_{I \setminus \{ j\}}\otimes \uZ_J  \to \uA_I \otimes \uZ_J \to \bZ_{p_j}\otimes \uA_{I \setminus \{ j\}} \otimes \uZ_J \to 0,
\end{myeq}
and that, the composite 
$$\bZ_{p_j}\otimes \uA_{I \setminus \{ j\}} \otimes \uZ_J \to \uA_I \otimes \uZ_J \to \bZ_{p_j}\otimes \uA_{I \setminus \{ j\}} \otimes \uZ_J $$ 
is given by multiplication by $p_j$. We first identify the cohomology of $\uHal(S^0;\bZ_p \boxtimes \uM)$ for $C_{n/p}$-Mackey functors $\uM$. This will allow us to compute $\uHal(S^0;  \uA_I \boxtimes \uZ_J)$ via an inductive process. 
  
\begin{mysubsection} {Identification of $\uHal(S^0;\bZ_p \boxtimes \uM)$} \label{bZtensM}
The direct computation of homotopy groups of fixed point spectra \cite[Proposition V.3.2]{MM02} shows that the fixed point  spectra of Eilenberg-MacLane spectra are Eilenberg-MacLane spectra. In fact we have 
$$\pi_m(H\uN)^K (WK/H) \cong \begin{cases} \uN(G/q^{-1}(H)) & \mbox{ if } ~ m= 0 \\ 
                                          0 & \mbox{if } m\neq 0\end{cases} $$ 
Applying this to $K=C_p$ and $\uN=\bZ_p\otimes \uM$ we obtain 
\begin{myeq} \label{fix_sp} 
H(\bZ_p \boxtimes \uM)^{C_p} \simeq H\uM ~~\mbox{ as } C_{\frac{n}{p}}\mbox{-spectra}.
\end{myeq}

\begin{notation} 
 For a $G$-spectrum $X$ and a divisor $d$ of $n$, let $i_d^*(X)$ be the $C_d$-spectrum obtained by the restriction  of $X$  along the inclusion $C_d \to G$.  
\end{notation}

We conclude easily using \cite[Proposition V.2.3]{MM02} that if $p\nmid d$, 
\begin{myeq}\label{Resd}
 i_d^*(H(\bZ_p \boxtimes \uM_{\frac{n}{p}})) \cong H(\downarrow^G_{C_d}(\bZp \boxtimes \uM)) =0.
\end{myeq}
 We now use the results above to compute $\uHal(S^0; \bZ_p \boxtimes \uM)$, which turns out to depend only on $\alpha^{C_p}.$ We write $d = \frac{n}{p}$, and prove
 \begin{prop}\label{bZM}
 For a $C_d$-Mackey functor $\uM$ and $\alpha \in RO(G),$ we have 
$$\uH^{\alpha}_G(S^0; \bZ_p \boxtimes \uM) \cong \uH^{\alpha^{C_p}}_{C_d}(S^0; \uM)\boxtimes \bZ_p.$$
 \end{prop}

  \begin{proof}
For $\alpha \in RO(G)$, $\alpha^{C_p}$ is a virtual $C_d$-representation, and we also think of it as a virtual $G$-representation via the map $q: G \to G/C_p \cong C_d$. We first observe that 
 $$\uHal(S^0; \bZ_p \boxtimes \uM) \cong \uH^{\alpha^{C_p}}_G(S^0; \bZ_p \boxtimes \uM).$$
 To see this, consider $r \mid n$ such that $p \nmid r.$ Using Proposition \ref{orbit} and \eqref{Resd}, we have 
$$\uHal({G/C_r}_+; \bZ_p \boxtimes \uM) \cong \uparrow^{G}_{C_r}\uH^{\alpha}_{C_r}(S^0; \downarrow^{G}_{C_r}\bZ_p \boxtimes\uM)=0.$$ 
Next, using the cofibre sequence \eqref{cofd}, we readily deduce  that $\uHal(S(\xi^r)_+; \bZ_p \boxtimes \uM)=0$. Now incorporating this value in the long exact sequence \eqref{exr}, we obtain 
$$\uHal(S^0; \bZ_p \boxtimes \uM) \cong \uH^{\alpha-\xi^r}_G(S^0; \bZ_p \boxtimes \uM).$$ 
Therefore, adding or subtracting a copy of $\xi^r$ does not change the value of $\uHal(S^0; \bZ_p \boxtimes \uM)$ if $p \nmid r$. 
    Observe now that $\alpha - \alpha^{C_p}$ is a $\Z$-linear combination of $\xi^r$ where $p \nmid r$. It follows that 
$\uHal(S^0; \bZ_p \boxtimes \uM) \cong \uH^{\alpha^{C_p}}_{G}(S^0; \bZ_p \boxtimes \uM)$. Finally, we prove 
$$\uH^{\alpha^{C_{p}}}_G(S^0; \bZ_p\boxtimes \uM)\cong \uH^{\alpha^{C_p}}_{C_d}(S^0; \uM)\boxtimes \bZ_p.$$ 
By \eqref{Resd}, both the sides of the above are $0$ at the orbit $G/C_r$ whenever $p\nmid r.$ The remaining orbits are of the form $q^*(C_d/K)$ for $K \leq C_d$. In this case, we compute 
\begin{align*}
 \uH^{\alpha^{C_p}}_G(S^0; \bZ_p\boxtimes \uM)(q^*(C_d/K)) &= \{S^{-\alpha^{C_p}} \wedge q^*(C_d/K)_+, H(\bZ_p\boxtimes \uM)  \}^G\\ 
&\cong \{q^*(S^{-\alpha^{C_p}} \wedge C_d/K_+), H(\bZ_p\boxtimes \uM)  \}^G\\
&\cong \{S^{-\alpha^{C_p}} \wedge C_d/K_+, H(\bZ_p\boxtimes \uM)^{C_p}  \}^{C_d}\\
&\cong \{S^{-\alpha^{C_p}} \wedge C_d/K_+, H\uM  \}^{C_d}\\ 
&= (\uH^{\alpha^{C_p}}_{C_d}(S^0; \uM)\boxtimes \bZ_p)(q^*(C_d/K)).
 \end{align*}
  In the above, the first equivalence stems from  the fact that $q^*(S^{-\alpha^{C_p}}) \cong S^{-\alpha^{C_p}}$, the second equivalence is from \cite[Propostion V.3.10]{MM02}, and the third equivalence follows from \eqref{fix_sp}. Observe also that the equivalence is functorial with respect to maps $C_d/K_1 \to C_d/K_2$ in the $C_d$-Burnside category. Hence the result follows.
 \end{proof}
\end{mysubsection}

\begin{mysubsection}{Computations in the non-zero case} \label{A-nz}
We compute $\uHal(S^0;\uA)$ using the short exact sequences \eqref{tensorex1} and \eqref{tensorex2}. We use Proposition \ref{bZM} to deduce
\begin{prop}\label{cohses}
Let $I$ be a subset $\uk$ and $j \in I$. For $\alpha \in RO(G)$, there is a short exact sequence \begin{myeq}\label{az_ses}
0 \to \uHal(S^0;\bZ_{p_j}\boxtimes \uA_{I \setminus \{ j\}}\boxtimes \uZ_J)  \to \uHal(S^0;\uA_I \boxtimes \uZ_J) \to \uHal(S^0;\uA_{I \setminus \{ j\}} \boxtimes \uZ_{J \cup \{j\}}) \to 0.
\end{myeq}
\end{prop}

\begin{proof}
The short exact sequences of Mackey functors \eqref{tensorex1} and \eqref{tensorex2} induce cohomology long exact sequences 
\begin{myeq}\label{cohex1}
\cdots  \uHal(S^0; \bZ_{p_j} \boxtimes \uA_{I \setminus \{ j\}}\boxtimes \uZ_J)  \to \uHal(S^0;\uA_I \boxtimes \uZ_J) \to \uHal(S^0;\uA_{I \setminus \{ j\}} \boxtimes \uZ_{J \cup \{j\}})\cdots
\end{myeq}
and 
$$\cdots \uHal(S^0;\uZ^*_{p_j}\boxtimes \uA_{I \setminus \{ j\}}\boxtimes \uZ_J)  \to \uHal(S^0;\uA_I \boxtimes \uZ_J) \to \uHal (S^0; \bZ_{p_j} \boxtimes \uA_{I \setminus \{ j\}} \boxtimes \uZ_{J \cup \{j\}}) \cdots$$ 
such that the composite 
$$\uHal(S^0;\bZ_{p_j}\boxtimes \uA_{I \setminus \{ j\}}\boxtimes \uZ_J)  \to \uHal(S^0;\uA_I \boxtimes \uZ_J) \to \uHal(S^0;\bZ_{p_j}\boxtimes \uA_{I \setminus \{ j\}}\boxtimes \uZ_J)$$ 
is given by multiplication by $p_j$. Using Proposition \ref{bZM} we have 
 $$\uHal(S^0;\bZ_{p_j}\boxtimes \uA_{I \setminus \{ j\}}\boxtimes \uZ_J) \cong \bZ_{p_j} \boxtimes \uH^{\alpha^{C_{p_j}}}_{C_\frac{n}{p_j}}(S^0; \uA_{I \setminus \{ j\}}\boxtimes \uZ_J).$$ 
It follows from  Proposition \ref{ptorsion} that the groups have no $p_j$-torsion.  Therefore, the map 
$$\uHal(S^0;\bZ_{p_j}\boxtimes \uA_{I \setminus \{ j\}}\boxtimes \uZ_J)  \to \uHal(S^0;\uA_I \boxtimes \uZ_J)$$ 
is injective for all $\alpha$. Hence, \eqref{cohex1} reduces to the short exact sequence \eqref{az_ses}.
\end{proof}
 
In fact, we know that multiplication by $p_j$ is an isomorphism on groups which have no $p_j$-torsion. Therefore, the proof of Proposition \ref{cohses} further implies
\begin{prop}\label{bZtors}
The torsion factors in $\uHal(S^0; \bZ_{p_j} \boxtimes \uA_{I\setminus \{j\}}\boxtimes \uZ_{J})$ are a summand of $\uHal(S^0;   \uA_I\boxtimes \uZ_{J})$. 
\end{prop}

We now use these results to complete the calculations in the non-zero case. In the theorem below we use the convention that 
\begin{myeq} \label{conv1} 
\uH^{\alpha}_{C_1}(S^0; \uZ) =\begin{cases} 0 & \mbox{ if } \alpha \neq 0 \\ 
                                                                         \Z  & \mbox{ if } \alpha = 0. \end{cases}
\end{myeq} 
This appears in the case $J=\varnothing$ below.  
 \begin{thm}\label{coh_ar}
 Suppose $I$ and $J$ as in Proposition \ref{cohses} and $\alpha \in RO(G)$ such that $|\alpha^H| \neq 0$  for all $H \leq C_I$. Then  
 $$\uHal(S^0; \uA_I \boxtimes \uZ_J) \cong \bigoplus\limits_{\II \subseteq I} \bZ_{\II} \boxtimes \uH^{\alpha^{C_\II}}_{C_{\frac{n}{\II}}}(S^0; \uZ).$$
 \end{thm}
 \begin{proof}
Note that the given non-zero fixed points condition implies that the groups in $\uHal(S^0;\uZ)$ are all torsion \eqref{Zval}. By induction on $\# I$, the cardinality of $I$, we deduce using Proposition \ref{cohses}, that the groups in $\uHal(S^0;\uA_I\boxtimes \uZ_J)$ are all torsion. 

We proceed to the proof of the theorem again by induction on $\# I$. For $\# I=1$ (that is $I=\{i\}$), we use Proposition \ref{bZtors} to deduce that the short exact sequences 
 $$0 \to \bZ_{\{i\}} \boxtimes  \uH^{\alpha^{C_{\{i\}}}}_{C_J}(S^0; \uZ_J) \to \uHal(S^0; \uA_{\{i\}}\boxtimes \uZ_J) \to \uHal(S^0; \uZ) \to 0$$  
are all split. 
In the general case, we know from the induction hypothesis  that 
$$\uHal(S^0;\uA_{I \setminus \{ j\}} \boxtimes \uZ_{J \cup \{j\}}) \cong \bigoplus\limits_{\II \subseteq I \setminus \{ j\}} \bZ_\II \boxtimes \uH^{\alpha^{C_\II}}_{C_{\frac{n}{\II}}}(S^0; \uZ_{\frac{n}{\II}}).$$ 
Proposition \ref{bZM} and the induction hypothesis together allow us to conclude 
$$\uH^{\alpha}_{G}(S^0; \bZ_{p_j}\boxtimes \uA_{I \setminus \{ j\}}\boxtimes \uZ_J) \cong \bigoplus\limits_{\II \subseteq I \setminus \{j \}}\bZ_{\II \cup \{j \}} \boxtimes \uH^{\alpha^{C_{\II \cup \{j \}}}}_{C_{\frac{n}{\II \cup \{j \}}}}(S^0; \uZ_{\frac{n}{\II \cup \{j \}}}).$$  
Now summing the two factors we conclude the result by induction. 
\end{proof}

\begin{exam}\label{compG}
In the case $G=C_p$, Theorem \ref{coh_ar} states that if $|\alpha|\neq 0$ and $|\alpha^{C_p}|\neq 0$, we have
$$\uH^\alpha_{C_p}(S^0;\uA) \cong \uH^\alpha_{C_p}(S^0;\uZ).$$

In the case $G=C_{pq}$, Theorem \ref{coh_ar} states that if $|\alpha|\neq 0$, $|\alpha^{C_p}|\neq 0$, $|\alpha^{C_q}|\neq 0$, and $|\alpha^{C_{pq}}|\neq 0$, we have
$$\uH^\alpha_{C_{pq}}(S^0;\uA) \cong \uH^\alpha_{C_{pq}}(S^0;\uZ) \oplus \bZ_p\boxtimes \uH^{\alpha^{C_p}}_{C_q}(S^0;\uZ) \oplus \bZ_q \boxtimes \uH^{\alpha^{C_q}}_{C_p}(S^0;\uZ).$$
\end{exam}

\vspace*{0.5cm}

An immediate observation from Theorem \ref{coh_ar} (and also using Theorem \ref{zercoh}) is the following Corollary. 
 \begin{cor}\label{Aodd}
 For $\alpha \in RO(G)$ odd, we have 
$$\uHal(S^0; \uA) \cong \bigoplus\limits_{\II \subseteq \uk} \bZ_{\II} \boxtimes \uH^{\alpha^{C_{\II}}}_{C_{\frac{n}{\II}}}(S^0; \uZ_{\frac{n}{\II}}).$$
If further $|\alpha^H| \leq 1$ for all $H\leq G$, $\uHal(S^0;\uA)=0$. 
 \end{cor}

\begin{exam}\label{Aoddneg}
The decomposition formula of $\uHal(S^0;\uA)$ in Corollary \ref{Aodd} for odd $\alpha$ may be stated in more explicit terms as (see Theorem \ref{non-zero})
$$\uHal(S^0;\uA) \cong \bigoplus\limits_{i=1}^k \bigoplus\limits_{\stackrel{\II\subset \uk\setminus\{i\},} {\stackrel{|\alpha^{C_\II}|<0,} {\stackrel{|\alpha^{C_{\II p_i}}|> 1}{ }}}}\bZ_\II\boxtimes \uZ_{\frac{n}{p\II}}\boxtimes \bZpi.$$ 
From this equation we see that $\uHal(S^0;\uA)=0$, if $\alpha$ is such that $|\alpha^{C_\II}| < 0$ implies $|\alpha^{C_{\II p}}|\leq 1$ for all $p \nmid |\II|$, which is slightly stronger than the condition in Corollary \ref{Aodd}. For example, in the case $n=pqr$ we may have the following fixed point dimensions for odd $\alpha$
$$ \left\{
\begin{array}{clclc} 
          &  |\alpha| < 0,           \\
|\alpha^{C_p}| =1,   &  |\alpha^{C_q}| <0,   &  |\alpha^{C_r}|=1,  \\
|\alpha^{C_{pr}}| >1, &  |\alpha^{C_{pq}}| =1,  & |\alpha^{C_{qr}}|=1,\\
                             & |\alpha^{C_{pqr}}| \mbox{ any odd value}
\end{array} \right \}, $$
where $\uHal(S^0;\uA)=0$. 
\end{exam}

\vspace*{0.5cm}

Another consequence of Theorem \ref{coh_ar} is 
\begin{cor}\label{Ators}
The torsion elements $a$ of $\uHal(S^0;\uA_I\boxtimes \uZ_J)$ satisfy $na =0$.  
\end{cor}

\begin{proof}
The torsion part of $\uHal(S^0;\uA_I\boxtimes \uZ_J)$ breaks up into a direct sum of the torsion part of $\uHal(S^0; \bZ_{p_j} \boxtimes \uA_{I\setminus \{j\}}\boxtimes \uZ_{J})$ and a subgroup of the torsion part of $\uHal(S^0;  \uA_{I\setminus \{j\}}\boxtimes \uZ_{J\cup \{j\}})$ by Proposition \ref{bZtors}. Now the result follows by induction and the computation with $\uZ$-coefficients.  
\end{proof}
\end{mysubsection}

\begin{mysubsection}{Computations in the mostly non-zero case} \label{A-mnz}
We now proceed to extend the arguments of Theorem \ref{coh_ar} to the mostly non-zero case. In this case the short exact sequences \eqref{az_ses} may not be split exact, however we work out the extensions to obtain a closed expression. It follows that the Mackey functor values obtained are determined by the fixed point dimensions of $\alpha$. For $\alpha$ which is not mostly non-zero, such a result cannot be true as observed in \cite{BG19} in the discussion following Theorem 6.5. 
 
\begin{notation} \label{perpnot}
Consider the Mackey functor $\uH^\alpha_{C_d}(S^0;\uZ)$ for some divisor $d$ of $n$ and $\alpha$ even. Assume in addition that $p$ is one of the $p_i$ which divides $d$. Note from Theorem \ref{non-zero} that $\uH^\alpha_{C_d}(S^0;\uZ)$ contains a copy of $\KK_i \bZpi$ if $|\alpha| > 0$ but $|\alpha^{C_p}|\leq 0$. We denote 
$$\uH^\alpha_{C_d}(S^0;\uZ)[p]= \begin{cases} \bigoplus\limits_{\stackrel{|\alpha^{C_{p_i}}|\leq 0,}{  p_i| d, ~ p_i \neq p}}\KK_i\bZpi & \mbox{ if } |\alpha| > 0 \\ 
    0 & \mbox{ if } |\alpha| \leq 0. \end{cases}  $$  
In other words, $\uH^\alpha_{C_d}(S^0;\uZ)[p]$ is obtained from $\uH^\alpha_{C_d}(S^0;\uZ)$ by removing a copy of $\KK_p\bZp$ if it was a summand. We may further extend the definition to $S \subset \uk$ consisting of divisors of $d$ as 
$$\uH^\alpha_{C_d}(S^0;\uZ)[S]= \begin{cases} \bigoplus\limits_{\stackrel{|\alpha^{C_{p_i}}|\leq 0,}  {p_i| d,~ i \notin S }}\KK_i\bZpi & \mbox{ if } |\alpha| > 0 \\ 
    0 & \mbox{ if } |\alpha| \leq 0. \end{cases}  $$  
That is,  $\uH^\alpha_{C_d}(S^0;\uZ)[S]$ is obtained from $\uH^\alpha_{C_d}(S^0;\uZ)$ by removing a copy of each $\KK_p\bZp$ for $p=p_i$ such that $i\in S$ if it was a summand. 
\end{notation}

\vspace*{0.5 cm}

In the following theorem we denote for $\II \subset \uk$ and $\alpha \in RO(G)$, 
$$\zeta_\alpha(\II) =\{ i \in \uk ~|~ i \notin \II, ~ |\alpha^{C_{\II p_i}}|=0\}.$$ 
In the case $J=\varnothing$, we use \eqref{conv1}.   
\begin{thm}\label{indzero}
Suppose that $\alpha$ is mostly non-zero. Then 
$$\uHal(S^0;\uA_I\boxtimes \uZ_J) = \bigoplus_{\II\subset I} \bZ_\II\boxtimes \uH^{\alpha^{C_\II}}_{C_{n/\II}}(S^0;\uZ)[\zeta_\alpha(\II) \cap I].$$
\end{thm}

\begin{proof}
We work the string of short exact sequences \eqref{az_ses} proceeding by induction on $\# I$. Assume by way of induction that the result holds for subsets of $\uk$ of cardinality less than $\# I$. Therefore we have (assuming $k\in I$)
\begin{myeq}\label{Zterm}
\uHal(S^0;\uA_{I\setminus\{k\}}\boxtimes \uZ_{J\cup \{k\}})   = \bigoplus_{\II\subset I\setminus\{k\}} \bZ_\II\boxtimes \uH^{\alpha^{C_\II}}_{C_{n/\II}}(S^0;\uZ)[\zeta_\alpha(\II)\cap I\setminus\{k\}],
\end{myeq}
and using Proposition \ref{bZM} that
$$\uHal(S^0;\bZ_{p_k}\boxtimes \uA_{I\setminus \{k\}}\boxtimes \uZ_J)  = \bigoplus_{\II \subset I\setminus \{k\}} \bZ_{\II p_k} \boxtimes \uH^{\alpha^{C_{\II p_k}}}_{C_{n/\II p_k}}(S^0;\uZ)[\zeta_\alpha(\II\cup \{k\})\cap I] .$$
In order to complete the proof, we observe that the Mackey functor $\uHal(S^0;\uA_I\boxtimes \uZ_J)$ has a summand isomorphic to $\uHal(S^0;\bZ_{p_k}\boxtimes \uA_{I\setminus\{k\}}\boxtimes \uZ_J)$, and a summand which maps isomorphically to the sub-factors of $\uHal(S^0;\uA_{I\setminus\{k\}}\boxtimes \uZ_{J\cup \{k\}})$ leaving out precisely the factors corresponding to $\bZ_\II\boxtimes \KK_{p_k}\bZpk$, where $\II\subset I\setminus \{k\}$ and $|\alpha^{C_{\II p_k}}|=0$.

We work the cohomology exact sequence associated to the coefficient sequence 
$$0 \to \uZ^*_{p_k}\boxtimes \uA_{I\setminus \{k\}}\boxtimes \uZ_J \to \uA_I \boxtimes \uZ_J \to \bZ_{p_k} \boxtimes \uA_{I\setminus \{k\}}\boxtimes \uZ_J \to 0$$
at the index $\alpha$ which is even. From Proposition \ref{bZtors} applied to $\alpha -1$, it follows that 
$$\uHal(S^0;\uZ^*_{p_k}\boxtimes \uA_{I\setminus \{k\}}\boxtimes \uZ_J) \to \uHal(S^0;\uA_I \boxtimes \uZ_J)$$
is injective. We prove that for $\alpha$ mostly non-zero, the inclusion is actually split injective.  We note the formula (Remark \ref{Z*cn})
$$\uHal(S^0;\uZ^*_{p_k}\boxtimes \uA_{I\setminus \{k\}}\boxtimes \uZ_J) \cong \uH^{\alpha+2-\xi^{\frac{n}{p_k}}}_G(S^0;\uZ_{p_k}\boxtimes \uA_{I\setminus \{k\}}\boxtimes \uZ_J),$$
  which implies that  $\uHal(S^0;\uZ^*_{p_k}\boxtimes \uA_{I\setminus \{k\}}\boxtimes \uZ_J) $ is isomorphic to the summand of $\uHal(S^0;\uA_{I\setminus \{k\}}\boxtimes \uZ_{J\cup \{k\}})$ complimentary to the factors 
$$\bigoplus\limits_{\stackrel{\II\subset I\setminus \{k\},}{|\alpha^{C_{\II p_k}}|=0, |\alpha^{C_\II}|>0} } \bZ_\II\boxtimes \KK_{p_k}\bZpk.$$ 
In fact, the composite
$$\uZ^*_{p_k}\boxtimes \uA_{I\setminus \{k\}}\boxtimes \uZ_J \to \uA_I \boxtimes \uZ_J \to \bZ_{p_k} \boxtimes \uA_{I\setminus \{k\}}\boxtimes \uZ_{J\cup \{k\}}$$ 
induces 
$$\uHal(S^0;\uZ^*_{p_k}\boxtimes \uA_{I\setminus \{k\}}\boxtimes \uZ_J) \to \uHal(S^0;\uA_{I\setminus \{k\}}\boxtimes \uZ_{J\cup \{k\}})$$
 which is an isomorphism at $G/H$ for $p_k \nmid |H|$, as 
$$\uHal(S^0;\bZ_{p_k}\boxtimes A_{I\setminus \{k\}} \boxtimes \uZ_J)(G/H)=0.$$ 
To prove that  $\uHal(S^0;\uZ^*_{p_k}\boxtimes \uA_{I\setminus \{k\}}\boxtimes \uZ_J) $ is a retract we note that each factor other than those of the form $ \bZ_\II\boxtimes \KK_{p_k}\bZpk$ are either of the form $\uZ_{p_k} \boxtimes \uM$, or $\uZ_{p_k}^\ast \boxtimes \uM$. The latter factors support a splitting map 
$$\uZ_{p_k} \boxtimes \uM \subset \uHal(S^0;\uZ^*_{p_k}\boxtimes \uA_{I\setminus \{k\}}\boxtimes \uZ_J) \to \uHal(S^0;\uA_{I\setminus \{k\}}\boxtimes \uZ_{J\cup \{k\}}) \to \uZ_{p_k}\boxtimes \uM,$$
with the composite being an isomorphism at $G/H$ as above and hence an isomorphism. The factors of the form $\uZ^\ast_{p_k} \boxtimes \uM$ support a retract by a similar argument. 
 The remaining factors are of the form $\bZ_\II\boxtimes \KK_{p_k}\bZpk$. Now we may apply induction, and deduce that this factor also carries a section unless $J=\varnothing$ and $\II=I\setminus \{k\}$. Otherwise, we may write the factor as $\CC_p\uN$ or $\KK_p\uN$ for some $p$ (among the $p_i$), and we deduce the splitting as in the above argument by changing the $p_k$. In the final case, the group is $\bZpk$, and this has a splitting as by Corollary \ref{Ators}, there is no higher $p_k$ torsion in $\uHal(S^0;\uA_I\boxtimes \uZ_J)$. 

Finally, it remains to prove that the kernel of 
$$\uHal(S^0;\bZ_{p_k}\boxtimes \uA_{I\setminus \{k\}}\boxtimes \uZ_J)  \to \uH^{\alpha+1}_G(S^0;\uZ^*_{p_k}\boxtimes \uA_{I\setminus \{k\}}\boxtimes \uZ_J)$$
is isomorphic to $\uHal(S^0;\bZ_{p_k}\boxtimes \uA_{I\setminus \{k\}}\boxtimes \uZ_J)$. For this, note that the torsion parts are all in the kernel by Proposition \ref{bZtors}. The torsion-free parts are of the form $\bZ_{\II p_k} \boxtimes \uH^{\alpha^{C_{\II p_k}}}_{C_{\frac{n}{\II p_k}}}(S^0;\uZ)$ with $|\alpha^{C_{\II p_k}}|=0$. Now the kernel of 
$$\uH^{\alpha+1}_G(S^0;\uZ^*_{p_k}\boxtimes \uA_{I\setminus \{k\}}\boxtimes \uZ_J) \to \uH^{\alpha+1}_G(S^0;\uA_I \boxtimes \uZ_J)$$
may only contain terms which are $p_k$-torsion, and hence, of the form $\bZ_{\II'}\boxtimes \KK_{p_k}\bZpk$. We finish the proof by observing (Proposition \ref{mapzI}) that a non-trivial map 
$$\bZ_{\II p_k} \boxtimes \uZ \boxtimes \uZ^\ast \to \bZ_{\II'}\boxtimes \KK_{p_k}\bZpk$$
exists if and only if $\II=\II'$ and in this case it is induced by the usual surjection $\Z \to \Z/p_k$. This kernel is clearly isomorphic to  $\bZ_{\II p_k} \boxtimes \uH^{\alpha^{C_{\II p_k}}}_{C_{\frac{n}{\II p_k}}}(S^0;\uZ)$.
\end{proof}

\vspace*{0.5cm}

An immediate corollary of Theorem \ref{indzero} is that in the case $|\alpha|=0$ and the other fixed points non-zero, the formula is not different from the non-zero case. 
\begin{cor}\label{dim0}
Let $\alpha \in RO(G)$ with $|\alpha|=0$ and $|\alpha^H|\neq 0$ for all $H\neq e$. Then, 
$$\uHal(S^0;\uA_I\boxtimes \uZ_J) = \bigoplus_{\II\subset I} \bZ_\II\boxtimes \uH^{\alpha^{C_\II}}_{C_{n/\II}}(S^0;\uZ).$$
\end{cor}

\begin{exam}
In the case $G=C_p$, the mostly non-zero cases are $|\alpha|=0$ or $|\alpha^{C_p}|=0$ but not both. For $|\alpha|=0$, Theorem \ref{indzero} says that 
$$\uH^\alpha_{C_p}(S^0;\uA) \cong \uH^\alpha_{C_p}(S^0;\uZ).$$
In the case $|\alpha^{C_p}|=0$, the computation implies 
$$\uH^\alpha_{C_p}(S^0;\uA) \cong \bZ.$$

The case $G=C_{pq}$ has the following cases of mostly non-zero $\alpha$ (using the symmetry between $p$ and $q$) : 
\begin{align*}
&1. ~(|\alpha|=0, |\alpha^H|\neq 0,~ H\neq e) \\
 &2. ~(|\alpha^{C_p}|=0, |\alpha^H|\neq 0,~ H\neq C_p), \\ 
&3. ~(|\alpha^{C_{pq}}|=0, |\alpha^H|\neq 0,~ H\neq C_{pq}),\\
 &4.~ (|\alpha|=0, |\alpha^{C_{pq}}|=0,~ |\alpha^{C_p}|\neq 0,~ |\alpha^{C_q}|\neq 0). 
\end{align*}  
For $\alpha$ as in $1$, Corollary \ref{dim0} implies (Example \ref{compG})
$$\uH^\alpha_{C_{pq}}(S^0;\uA) \cong  \uH^\alpha_{C_{pq}}(S^0;\uZ) \oplus \bZ_p\boxtimes \uH^{\alpha^{C_p}}_{C_q}(S^0;\uZ) \oplus \bZ_q \boxtimes \uH^{\alpha^{C_q}}_{C_p}(S^0;\uZ).$$
For $\alpha$ as in $4$ or as in $3$, Theorem \ref{indzero} implies 
$$\uH^\alpha_{C_{pq}}(S^0;\uA) \cong \uH^\alpha_{C_{pq}}(S^0;\uZ) \oplus \bbZ,$$
with $\bbZ \cong \bZ_p\boxtimes \bZ_q$. For $\alpha$ as in $2$, Theorem \ref{indzero} implies 
$$\uH^\alpha_{C_{pq}}(S^0;\uA) \cong  \uH^\alpha_{C_{pq}}(S^0;\uZ)[p] \oplus \bZ_p\boxtimes \uH^{\alpha^{C_p}}_{C_q}(S^0;\uZ) \oplus \bZ_q \boxtimes \uH^{\alpha^{C_q}}_{C_p}(S^0;\uZ),$$
with 
$$\uH^\alpha_{C_{pq}}(S^0;\uZ)[p] \cong \begin{cases} \KK_q\bZq & \mbox{ if } |\alpha|>0, |\alpha^{C_q}|<0 \\ 
0 & \mbox{ otherwise}. \end{cases}$$
\end{exam}

\end{mysubsection}

\begin{mysubsection}{The groups $\tHal(S^0;\uA)$ in cases with many zeros} \label{A-mz}
The computation of Mackey functor valued cohomology $\uHal(S^0;\uA)$ in cases with many zeros,  depends on the actual expression of $\alpha$, and not just it's fixed points. However, the groups $\tHal(S^0;\uA)$ turn out to depend only on the fixed points. We proceed to compute these finitely generated Abelian groups by describing the torsion-free part and the torsion part separately. 

\begin{thm}\label{mz-free} 
The torsion free part of $\tHal(S^0;\uA_I\boxtimes \uZ_J)$ is isomorphic to the free Abelian group of rank $\#\{\II \subset I | |\alpha^{C_\II} | =0 \}$\footnote{We assume $\alpha^{C_\varnothing}=\alpha$.}.  
\end{thm}  

\begin{proof} 
In order to compute the torsion-free part, we may tensor with $\Q$. From the short exact sequence \eqref{az_ses} we have 
\begin{align*} 
\dim(\uHal(S^0;\uA_I \boxtimes \uZ_J) \otimes \Q) & = \dim( \uHal(S^0;\bZ_{p_j}\boxtimes \uA_{I \setminus \{ j\}}\boxtimes \uZ_J)  \otimes \Q) \\ 
& + \dim ( \uHal(S^0;\uA_{I \setminus \{ j\}} \boxtimes \uZ_{J \cup \{j\}}) \otimes \Q).
\end{align*} 
We may now proceed by induction on $\# I$ noting that the case $I=\varnothing$ follows from \eqref{Zval}. From the induction hypothesis we have 
$$ \dim ( \uHal(S^0;\uA_{I \setminus \{ j\}} \boxtimes \uZ_{J \cup \{j\}}) \otimes \Q) = \# \{ \II \subset I \setminus \{j \} | ~ |\alpha^{C_\II}|=0\},$$
and from Proposition \ref{bZM},
\begin{align*}
 \dim( \uHal(S^0;\bZ_{p_j}\boxtimes \uA_{I \setminus \{ j\}}\boxtimes \uZ_J)  \otimes \Q) & = \dim( \uH^{\alpha^{C_{p_j}}}_{C_{n/p_j}}(S^0;\uA_{I \setminus \{ j\}}\boxtimes \uZ_J)  \otimes \Q) \\ 
& =  \# \{ \II \subset I \setminus \{j \} | ~ |\alpha^{C_{\II p_j}}|=0\}.
\end{align*}
The result follows from induction. 
\end{proof}

We now complete the calculation by determining the torsion part of $\tHal(S^0;\uA_I\boxtimes \uZ_J)$, that is, by determining the $p$-torsion for every prime divisor $p=p_i$ of $n$. We apply Corollary \ref{Ators} to conclude that the $p$-torsion is a direct sum of copies of $\Z/p$. We denote 
$$\nu_p^{I,J}(\alpha) = \begin{cases} 
                                  \#\{\II \in I | ~ |\alpha^{C_\II}| >0, |\alpha^{C_{\II p}}| \leq 0 \} & \mbox{ if } i \in J \\ 
                                  \#\{\II \in I | ~ |\alpha^{C_\II}| >0, |\alpha^{C_{\II p}}| < 0 \} & \mbox{ if } i \notin J,\\ 
\end{cases}$$
and that $\nu_p^{I,J}(\alpha)=0$ if $p\nmid n$. In terms of this notation we prove, 
\begin{thm} \label{mz-tors}
For a prime $p$ dividing $n$, the $p$-torsion of $\tHal(S^0;\uA_I\boxtimes \uZ_J)$ is $(\Z/p)^{\nu_p^{I,J}(\alpha)}$. 
\end{thm}

\begin{proof}
The theorem follows for $I=\varnothing$ from \eqref{Zval}, and also for non-zero $\alpha$ from Theorem \ref{coh_ar}, so it suffices to consider $\alpha$ even. We consider the Mackey functor short exact sequence \eqref{tensorex2} which induces the long exact sequence 
$$\cdots \to \uHal(S^0; \uZ^\ast_{p_j} \boxtimes \uA_{I\setminus \{j\}} \boxtimes \uZ_J) \to \uHal(S^0; \uA_I \boxtimes \uZ_J) \to \uHal(S^0; \bZ_{p_j} \boxtimes \uA_{I\setminus \{j\}} \boxtimes \uZ_J) \to \cdots. $$
Since $\alpha -1 $ is odd, $\uH^{\alpha -1}_G (S^0; \bZ_{p_j} \boxtimes \uA_{I\setminus \{j\}} \boxtimes \uZ_J)$ has no free summand. We now apply Proposition \ref{bZtors} to conclude that 
$$\uHal(S^0; \uZ^\ast_{p_j} \boxtimes \uA_{I\setminus \{j\}} \boxtimes \uZ_J) \to \uHal(S^0; \uA_I \boxtimes \uZ_J)$$ maps the $p$-torsion injectively, and that the map 
$$\uHal(S^0; \uA_I \boxtimes \uZ_J) \to \uHal(S^0; \bZ_{p_j} \boxtimes \uA_{I\setminus \{j\}} \boxtimes \uZ_J) $$
 splits at the $p$-torsion subgroup. Therefore, we have 
\begin{align*} 
p\mbox{-torsion}( \uHal(S^0; \uA_I \boxtimes \uZ_J)) &= p\mbox{-torsion}(\uHal(S^0; \uZ^\ast_{p_j} \boxtimes \uA_{I\setminus \{j\}} \boxtimes \uZ_J)) \\ 
   & \bigoplus p\mbox{-torsion}(\uHal(S^0; \bZ_{p_j} \boxtimes \uA_{I\setminus \{j\}} \boxtimes \uZ_J) ). 
\end{align*} 
We also note that (Remark \ref{Z*cn})
$$\uHal(S^0; \uZ^\ast_{p_j} \boxtimes \uA_{I\setminus \{j\}} \boxtimes \uZ_J) \cong \uH^{\alpha +2 - \xi^{\frac{n}{p_j}}}_G (S^0; \uA_{I\setminus \{j\}} \boxtimes \uZ_{J\cup \{j\}})$$
and (Proposition \ref{bZM})
$$\tHal(S^0; \bZ_{p_j} \boxtimes \uA_{I\setminus \{j\}} \boxtimes \uZ_J) \cong \tH^{\alpha^{C_{p_j}}}_{C_{\frac{n}{p_j}}}(S^0;  \uA_{I\setminus \{j\}} \boxtimes \uZ_J) .$$
Now we proceed by induction on $\# I$, and realize that it suffices to prove 
$$\nu_p^{I,J}(\alpha) =\nu_p^{I\setminus\{j\}, J \cup \{j\}}(\alpha+2-\xi^{\frac{n}{p_j}}) + \nu_p^{I\setminus \{j\},J}(\alpha^{C_{p_j}}).$$
We have $p=p_i$ for some $i$, and that $i$ can occur in three different ways : 1) $i\in J$, 2) $i\in I\setminus \{j\}$, and 3) $i=j$. In case 1),  for $\II \subset I-\{j\}$, 
$$\alpha^{C_\II} = (\alpha+2-\xi^{\frac{n}{p_j}})^{C_\II},~ ~~\alpha^{C_{\II p}} = (\alpha+2-\xi^{\frac{n}{p_j}})^{C_{\II p}},$$
 and
\begin{align*}
\{\II \subset I |~|\alpha^{C_\II}| > 0,~|\alpha^{C_{\II p}}|\leq 0 \} & = \{\II \subset I\setminus \{j\} |~|\alpha^{C_\II}| > 0,~|\alpha^{C_{\II p}}|\leq 0 \} \\ 
                        &  \coprod \{\II \subset I \setminus \{j\} |~|\alpha^{C_{\II p_j}}\}| > 0,~|\alpha^{C_{\II p_j p}}|\leq 0 \}.
\end{align*}
It follows that 
\begin{align*}
& \nu_p^{I\setminus\{j\}, J \cup \{j\}}(\alpha+2-\xi^{\frac{n}{p_j}}) + \nu_p^{I\setminus \{j\},J}(\alpha^{C_{p_j}}) \\
 & =  \nu_p^{I\setminus\{j\}, J \cup \{j\}}(\alpha) + \nu_p^{I\setminus \{j\},J}(\alpha^{C_{p_j}}) \\
& =\#  \{\II \subset I\setminus \{j\} |~|\alpha^{C_\II}| > 0,~|\alpha^{C_{\II p}}|\leq 0 \}  + \#  \{\II \subset I \setminus \{j\} |~|\alpha^{C_{\II p_j}}\}| > 0,~|\alpha^{C_{\II p_j p}}|\leq 0 \} \\ 
& = \# \{\II \subset I |~|\alpha^{C_\II}| > 0,~|\alpha^{C_{\II p}}|\leq 0 \} \\
& = \nu_p^{I,J}(\alpha).
\end{align*}

The case 2) is almost the same as case 1) replacing $|\alpha^{C_\II}|\leq 0$ with $|\alpha^{C_\II}|<0$. Finally in case 3), we have $p=p_j$ and $ \nu_p^{I\setminus \{j\},J}(\alpha^{C_{p_j}}) =0$. We also have $\II \subset I-\{j\}$, 
$$\alpha^{C_\II} = (\alpha+2-\xi^{\frac{n}{p_j}})^{C_\II},~~ \alpha^{C_{\II p}} + 2= (\alpha+2-\xi^{\frac{n}{p_j}})^{C_{\II p}}, $$ 
so that, 
\begin{align*}
& \{\II \subset I |~|\alpha^{C_\II}| > 0,~|\alpha^{C_{\II p}}|< 0 \} \\ 
& = \{\II \subset I\setminus \{j\} |~|\alpha^{C_\II}| > 0,~|\alpha^{C_{\II p}}+2|\leq 0 \} \\ 
                         & = \{\II \subset I \setminus \{j\} |~|(\alpha+2-\xi^{\frac{n}{p_j}})^{C_\II}\}| > 0,~|(\alpha + 2 - \xi^{\frac{n}{p_j}}|^{C_{\II p}}|\leq 0 \}.
\end{align*}
It follows that $\nu_p^{I\setminus\{j\},J\cup \{j\}}(\alpha + 2 - \xi^{\frac{n}{p_j}})= \nu_p^{I,J}(\alpha)$. 
\end{proof}
We summarize the results of Theorem \ref{mz-free} and Theorem \ref{mz-tors} in the following corollary.
\begin{cor}\label{mzcalc}
For $\alpha \in RO(G)$,
$$\tHal(S^0;\uA_I\boxtimes \uZ_J) \cong \Z^{\#\{\II \subset I |~ |\alpha^{C_\II}|=0 \}} \oplus \bigoplus_i (\Z/p_i)^{\nu_{p_i}^{I,J}(\alpha)}. $$ 
\end{cor}
\end{mysubsection}

\section{$G$-spaces that have free cohomology}\label{freenessthm}

The description of $\uHbs(S^0;\uA)$ may be used to prove structural results for the cohomology of $G$-spaces $X$. In this section, we derive sufficient conditions under which the cohomology becomes a free module over $\uHbs(S^0; \uA)$. The approach followed here is a generalization of the results of Lewis \cite{Lew88} for the group $C_p$. 

The spaces which we consider are formed by attaching disks in $G$-representations via identification maps on their boundaries. These were defined as {\it generalized cell complexes} in \cite{FL04}. Such cell complexes arise naturally in equivariant Morse theory \cite{Was69}. We briefly recall their definition. 
\begin{defn} 
A generalized $G$-cell complex $X$ is a $G$-space with a filtration $\{X_n\}_{n\geq 0}$ of subspaces such  that \\
(a) $X_0$ is a finite $G$-set. \\ 
(b) For each $n$, $X_{n+1}$ is formed from $X_n$ by attaching cells of the form $G\times_H D(V)$ where $H\leq G$ and $V$ is an $H$-representation. \\ 
(c) $X = \cup_n X_n$ has the colimit topology. 
\end{defn}

The $G$-sphere $S^V$ may be written as $G/G \cup G\times_G D(V)$, so it is a generalized cell complex. It is also clear that homotopy cofibres of maps between representation spheres are generalized cell complexes. Since the cohomology of $S^V$ is a $V$-fold suspension of $\uHbs(S^0;\uA)$, it is free on a generator in degree $V$. Thus, it is natural to explore the free cohomology question among these complexes. We prove that if the cells are even dimensional, and if the cell attachments satisfy some condition (``{\it even type}'' defined below), then the cohomology is a free module.  In the following definition note that for $H\leq G$ (which is an inclusion of cyclic groups), $RO(G) \to RO(H)$ is surjective, so, all the cells are of the form $G\times_H D(V)$ for $V\in RO(G)$. 
\begin{defn}
a) Let $V$ be a $G$-representation. A cell of type $G \times_H D(V)$ is called even if $V$ is an even element in $RO(G)$. \\
b) Let $V$ and $W$ be two $G$-representations. We say $W \ll V$ if $|W^S| < |V^S|$ for some subgroup $S$ of $G$ implies $|W^T|\leq |V^T|$ for all subgroups $T$ containing $S$. \\ 
c) A generalized $G$-cell complex $X$ is said to be of {\it even type} if every cell is even, and if the cell $G \times_H D(W)$ is attached before $G \times_H D(V)$, then $W \ll V$. 
\end{defn}

The freeness of the cohomology is proved inductively by showing that all the attaching maps induce the zero map on cohomology. This is implied by the two cell case which we prove in the lemma below. 
\begin{lemma}\label{ll}
If $W \ll V,$ any $\uHbs(S^0)$-module map 
$$\uH^{\bs - W}_G(G/K_+;\uA) \to \uH^{\bs - 1+V}_{G}(G/K^{\prime}_+;\uA)$$ 
is zero.
\end{lemma}

\begin{proof}
The proof follows the method of \cite[Lemma 8.3]{BG19}. If both $K$ and $K^{\prime}$ are $G$, then the $\uHbs(S^0;\uA)$-module map $\uH^{\bs - W}_G(S^0;\uA) \to \uH^{\bs + 1-V}_G(S^0;\uA)$ is determined by an element of group $\tH^{W + 1-V}_G(S^0;\uA)$. Now note that the given conditions imply that $|W+1-V|$ is odd, and its fixed point dimensions are $\leq 1$. Therefore, the cohomology group is zero by Corollary \ref{Aodd}. 

We next consider $K =G$, so that a $\uHbs(S^0;\uA)$-module map $\uH^{\bs - W}_G(S^0;\uA) \to \uH^{\bs +1-V}_G(G/K^{\prime}_+;\uA)$ is determined by the image of $1$. We compute 
$$\tH^{W +1-V}_G(G/K^{\prime}_+;\uA) \cong \tH^{W +1-V}_{K^{\prime}}(S^0;\uA),$$ 
which is zero for all subgroups $K^{\prime}$ of $G$ by Corollary \ref{Aodd}.

Finally, it remains to prove the case where $K$ is a proper subgroups of $G$. The  $\uHbs(S^0;\uA)$-module map 
$$\uH^{\bs -W}_G(G/{C_d}_+;\uA) \to \uH^{\bs +1-V}_G(G/{C_s}_+;\uA),$$ 
yields (and is determined by)  a $\uH^{\bs}_{C_d}(S^0;\uA)$-module map (via restriction to orbits of the form $G\times_{C_d} C_d/K$ for $K\leq C_d$)
$$\uH^{\bs -W}_{C_d}(S^0;\uA) \to \uH^{\bs +1-V}_{C_d}(G/{C_s}_+;\uA) \cong \bigoplus_{\frac{n\gcd(d,s)}{ds}}\uH^{\bs +1-V}_{C_{\gcd(d,s)}}(S^0;\uA).$$ 
The last equivalence comes from the fact that $G/C_s$ decomposes as a $C_d$-set into a disjoint union of $\frac{n\gcd(d,s)}{ds}$ copies of $C_d/C_{\gcd(d,s)}$. It follows that this map is determined by a tuple of elements in the group $\uH^{W +1-V}_{C_{(d,s)}}(S^0;\uA)$ and which is zero by Corollary \ref{Aodd}. Hence the result follows.
\end{proof}

Lemma \ref{ll} is now used to prove the freeness result by concluding the attaching maps induce the zero map on cohomology. The following theorem is proved in exactly the same manner as \cite[Theorem 8.4]{BG19}.
\begin{thm}\label{free}
Let $X$ be a generalized $G$-cell complex of even type in which the spaces $X_n$ of the filtration have finitely many cells. 
Then $\uHbs(X_+;\uA)$ is a free $\uHbs(S^0;\uA)$-module with summands consisting of  $\uHbs ({X_0}_+;\uA)$ and one copy of $\Sigma^{V} \uHbs({G/K}_+;\uA)$ for each cell of type $G_+ \wedge _K D(V)$ occurring in the cell structure of $X$.
\end{thm}

\begin{rmk}
Using Example \ref{Aoddneg} instead of Corollary \ref{Aodd}, we see that the condition $W\ll V$ may be weakened slightly to $W\tilde{\ll} V$ which states $|W^{C_\II}|< |V^{C_\II}|$ implies $|W^{C_{\II p_j}}| \leq |V^{C_{\II p_j}}|$ for all $j \notin \II$. 
\end{rmk}

Theorem \ref{free} has applications to the cohomology of certain equivariant complex projective spaces and Grassmannians. We refer to \cite[Section 8.1]{BG19} for explicit $G$-cell complex structures on these. We write $\UU(m)= \oplus_{r=0}^m \xi^r$. 
The complex projective space $\C P(\UU(m))$ is a generalized cell complex with a cell  $D(W_r)$ for each $0\leq r \leq m$ and $W_r= \xi^{-r} \UU(r-1)$. These cells are even dimensional and have fixed point dimensions 
$|W_r^{C_\II}| = 2\lfloor \frac{r}{|\II|}\rfloor,$
and therefore, are of even type. Hence we have 
\begin{thm}\label{cohcp}
For $1\leq m \leq \infty$, the cohomology $\uHbs(\C P(\UU(m));\uA)$ is isomorphic to $\oplus_{r=0}^n \uH^{\bs-W_r}_G(S^0 ; \uA)$.  
\end{thm}
We may also consider the complex Grassmannian $G(\UU(l),m)$ which has a Schubert cell decomposition \cite[Chapter 7]{FL04} with cells $D(W_{\ua})$ for Sch\"{u}bert symbols $\ua=(0\leq a_1\leq a_2\leq \cdots \leq a_m \leq l-m)$ and 
$$W_{\ua}= \bigoplus_{i=1}^m \bigoplus_{\stackrel{j=1}{j\notin \{a_1+1,\cdots, a_{i-1}+i-1\}}}^{a_i+i-1} \xi^{-(a_i+i)}\otimes \xi^j. $$
The fixed point dimensions of $W_{\ua}$ are given by 
$$|W_{\ua}^{C_\II}| = 2 \sum_{i=1}^m \lfloor \frac{a_i}{|\II|}\rfloor.$$ 
Therefore, $G(\UU(l),m)$ is also of even type, and we have 
\begin{thm}\label{cohgr}
For $1\leq l \leq \infty$ and $m\leq l$, the cohomology $\uHbs(G(\UU(l),m);\uA)$ is isomorphic to a free $\uHbs(S^0;\uA)$-module with one generator each in dimension $W_{\ua}$ for each Sch\"{u}bert symbol $\ua$.  
\end{thm}

\vspace*{.5cm}

\noindent{\small S}{\scriptsize TAT- }{\small M}{\scriptsize ATH }{\small U}{\scriptsize NIT,} {\small I}{\scriptsize NDIAN }{\small S}{\scriptsize TISTICAL }{\small I}{\scriptsize NSTITUTE, }{\small K}{\scriptsize OLKATA }{\footnotesize 700108, }{\small I}{\scriptsize NDIA}.\\
{\it E-mail address} : \texttt{samik.basu2@gmail.com}, \,\,\texttt{samikbasu@isical.ac.in}\\[0.2cm]

\noindent{\small D}{\scriptsize EPARTMENT OF }{\small M}{\scriptsize ATHEMATICS, }{\small U}{\scriptsize NIVERSITY OF }{\small H}{\scriptsize AIFA, }{\small H}{\scriptsize AIFA }{\footnotesize 3498838, }{\small I}{\scriptsize SRAEL.}\\
{\it E-mail address} : \texttt{surojitghosh89@gmail.com},\,\,\,\,\texttt{surojit@math.haifa.ac.il}\\[0.2cm]

\mbox{ }\\

\end{document}